\documentclass[11pt, a4paper]{amsart}
\usepackage[latin1]{inputenc}
\usepackage{amssymb,amsmath,amsthm}
\usepackage{enumerate}
\usepackage{xcolor}
\usepackage{nicefrac}

\usepackage{cases}
\usepackage{hyperref}

\usepackage[mathcal]{eucal}

\newcommand{\N}{\mathbb N}
\newcommand{\Z}{{\mathbb Z}} 
\newcommand{\F}{{\mathbb F}}

\DeclareMathOperator{\hdim}{hdim}
\DeclareMathOperator{\hspec}{hspec}
\DeclareMathOperator{\dens}{den}
\DeclareMathOperator{\dspec}{dspec}
\DeclareMathOperator{\dd}{d}

\DeclareMathOperator{\RG}{RG}
\DeclareMathOperator{\wt}{wt}

\newtheorem{Thm}{Theorem}[section]
\newtheorem{Prop}[Thm]{Proposition}
\newtheorem{Cor}[Thm]{Corollary}
\newtheorem{Lem}[Thm]{Lemma}

\newtheorem*{Thm-no-nr}{Interval Theorem}

\theoremstyle{definition}
\newtheorem{Defi}[Thm]{Definition}
\newtheorem{Rmk}[Thm]{Remark}
\newtheorem{Qu}[Thm]{Question}

\newtheorem*{Defi*}{Definition}

\theoremstyle{remark}
\newtheorem*{Not}{Notation}
\newtheorem{Ex}[Thm]{Example}

\numberwithin{equation}{section}

\title[Pro-$p$ groups and Hausdorff dimension]{Pro-$p$ groups of
  positive rank gradient and Hausdorff dimension}



\author[O.\ Garaialde Oca\~na]{Oihana Garaialde Oca\~na}

\address{Matematika Saila,
Euskal Herriko Unibertsitatearen Zientzia eta Teknologia Fakultatea,
 posta-kutxa 644, 48080 Bilbo, Spain}
\email{oihana.garayalde@ehu.eus}

\author[A.\ Garrido]{Alejandra Garrido}
\address{School of Mathematical \& Physical Sciences\\
	University of Newcastle\\
	University Drive\\
	Callaghan, 2308, NSW, Australia}
\email{alejandra.garrido@newcastle.edu.au}

\author[B.\ Klopsch]{Benjamin Klopsch}
\address{Heinrich-Heine-Universit\"{a}t D\"{u}sseldorf\\
	Mathematisches Institut\\
	Universit\"{a}tsstr. 1\\
	40225\\
	D\"{u}sseldorf, Germany}
\email{klopsch@math.uni-duesseldorf.de}

\thanks{The first author was supported by the Basque Government
  postdoctoral fellowship POS--2017--2--0031. The second author
  gratefully acknowledges the support of an Alexander von Humboldt
  Fellowship.  The research was partially conducted in the framework of the
  DFG-funded research training group ``GRK 2240: Algebro-Geometric
  Methods in Algebra, Arithmetic and Topology''.}

\begin{document}
	
\begin{abstract}
  Let $G$ be a finitely generated pro-$p$ group of positive rank
  gradient.  Motivated by the study of Hausdorff dimension, we show
  that finitely generated closed subgroups $H$ of infinite index in
  $G$ never contain any infinite subgroups $K$ that are subnormal
  in~$G$ via finitely generated successive quotients.  This pro-$p$
  version of a well-known theorem of Greenberg generalises similar
  assertions that were known to hold for non-abelian free pro-$p$
  groups, non-soluble Demushkin pro-$p$ groups and other related
  pro-$p$ groups.  The result we prove is reminiscent of Gaboriau's
  theorem for countable groups with positive first $\ell^2$-Betti
  number, but not quite a direct analogue.
		
  The approach via the notion of Hausdorff dimension in pro-$p$ groups
  also leads to our main results.  We show that every finitely
  generated pro-$p$ group $G$ of positive rank gradient has full
  Hausdorff spectrum $\hspec^\mathcal{F}(G) = [0,1]$ with respect to
  the Frattini series~$\mathcal{F}$.  Using different, Lie-theoretic
  techniques we also prove that finitely generated
    non-abelian free pro-$p$ groups and non-soluble Demushkin groups
    $G$ have full Hausdorff spectrum $\hspec^\mathcal{Z}(G) = [0,1]$
    with respect to the Zassenhaus series~$\mathcal{Z}$.  This
  resolves a long-standing problem in the subject of Hausdorff
  dimensions in pro-$p$ groups.

  In fact, the results about full Hausdorff spectra hold more
  generally for finite direct products of finitely generated pro-$p$
  groups of positive rank gradient and  for mixed finite
    direct products of finitely generated non-abelian free pro-$p$
    groups and non-soluble Demushkin groups, respectively.  The
  results with respect to the Frattini series generalise further to
  Hausdorff dimension functions with respect to arbitrary iterated
  verbal filtrations, for mixed finite direct products of finitely generated
  non-abelian free pro-$p$ groups and non-soluble Demushkin pro-$p$
  groups.

 Finally, we determine the normal Hausdorff spectra of
  such direct products.
\end{abstract}
	
\maketitle

	
\section{Introduction}  \label{sec:introduction}

We study the subgroup structure of finitely generated pro-$p$ groups
that resemble free pro-$p$ groups.  Here and throughout, $p$ denotes a
prime number and, as usual in the context of profinite groups,
subgroups are tacitly assumed to be closed. To make this
  explicit, we write $H \le_\mathrm{c} G$, respectively
$H \le_\mathrm{o} G$, to indicate that a subgroup $H$ is closed,
respectively open, in~$G$. Furthermore, group-theoretic notions, such
as generating properties, are considered topologically.

The classical Schreier index formula motivates the following notion,
first introduced in~\cite{La05}.  The (absolute) \emph{rank gradient}
of a finitely generated pro-$p$ group $G$ is
\[
\RG(G) = \inf_{H \le_\mathrm{o} G} \frac{\dd(H)-1}{\lvert G : H
  \rvert},
\]
where $\dd(H)$ denotes the minimal number of generators of a pro-$p$
group $H$ and the infimum is taken over all open subgroups
$H \le_\mathrm{o} G$.  We recall that in the present context $\RG(G)$
coincides with the (relative) \emph{rank gradient}
\[
\RG(G,\mathcal{S}) = \lim_{i \to \infty} \frac{\dd(G_i)-1}{\lvert G :
  G_i \rvert}
\]
with respect to any descending chain
$\mathcal{S} \colon G \supseteq G_1 \supseteq G_2 \supseteq \ldots$ of
open subgroups that provide a base of neighbourhoods of the identity
element; the limit exists, because the sequence of rational numbers is
monotone decreasing and bounded.  In practice, it is convenient to
work with descending chains of open normal subgroups $G_i$,
$i \in \N_0$, that intersect in the identity; we call such chains
\emph{filtration series}, or filtrations for short.
	
The class of pro-$p$ groups of positive rank gradient includes, for
instance:
\begin{itemize}
\item finitely generated non-abelian free pro-$p$ groups;
\item non-soluble Demushkin pro-$p$ groups, which occur
         as analogues of surface groups and as Galois groups of
         maximal $p$-extensions of local fields containing $p$th roots of
         unity (compare~\cite{Se95});
\item pro-$p$ analogues of limit groups, defined and studied in
        \cite{KoZa11};
\item free pro-$p$ products of finitely many non-trivial finitely generated
        pro-$p$ groups, with at least two factors and excluding
        $C_2 \amalg C_2$ (compare \cite[Thm.~2.3]{Pa15});
\item pro-$p$ completions of groups of positive $p$-deficiency, as
       constructed in~\cite{SP12} or arising from similar constructions
       in~\cite[Cor.~1.2]{Os11};
\item pro-$p$ completions of certain Golod--Shafarevich groups, as
        seen in~\cite{BuTh11}.
\end{itemize}

Recall that $K \le_\mathrm{c} G$ is \emph{subnormal} in $G$ if there
exists a finite descending sequence of subgroups
$G = K_0 \trianglerighteq K_1 \trianglerighteq \ldots \trianglerighteq
K_s = K$,
with $K_i$ normal in $K_{i-1}$ for $i \in \{1,\ldots, s\}$.  A
well-known theorem of Greenberg~\cite{Gr60} implies that a finitely
generated subgroup $\Delta$ of a finitely generated, non-abelian free
discrete group $\Gamma$ has finite index in $\Gamma$ if and only if
$\Delta$ contains a non-trivial subnormal subgroup of~$\Gamma$.  Our
first result is a version of Greenberg's theorem in the context of pro-$p$ groups.
	 
\begin{Defi*}
  We say that a subgroup $K \le_\mathrm{c} G$ is \emph{subnormal via
    finitely generated successive quotients} if there is a subnormal
  chain
  $G = K_0\trianglerighteq K_1\trianglerighteq \ldots \trianglerighteq
  K_s = K$
  such that each factor $K_{i-1}/K_i$, for $i \in \{1,\ldots,s\}$, is
  finitely generated.
\end{Defi*}

\begin{Thm}\label{Thm:posRG_normal_sgps_inf_gen}
  Let $G$ be a finitely generated pro-$p$ group of positive rank
  gradient, and let $H \le_\mathrm{c} G$ be a finitely generated
  subgroup.  Then the following are equivalent:
  \begin{enumerate}[\rm (1)]
  \item $H$ contains an infinite subgroup of~$G$ that is subnormal via
    finitely generated successive quotients;
  \item $H$ contains an infinite normal subgroup of~$G$;
  \item $H$ is open in~$G$.  
  \end{enumerate}
  In particular, every subgroup $K \le_\mathrm{c} G$ that is subnormal
  via finitely generated successive quotients satisfies: $K$ is
  finitely generated if and only if $K$ is either finite or open.
\end{Thm}
		
Theorem~\ref{Thm:posRG_normal_sgps_inf_gen} generalises, for instance,
results of Lubotzky~\cite[Prop.~3.3]{Lu82}, Kochloukova
\cite[Prop.~3.13]{Ko13}, and Snopce and
Zalesskii~\cite[Thm.~3.1]{SnZa16}.  We thank Mark Shusterman for
alerting us to his work on pro-$p$ groups of hereditarily linearly
increasing rank, where he already obtains closely related pro-$p$
versions of Greenberg's theorem; compare \cite[Thm.~2.7 and
Cor.~2.8]{Sh17}.  While we discovered
Theorem~\ref{Thm:posRG_normal_sgps_inf_gen} in a different way, it can
also easily be proved by adapting Shusterman's approach, the key
ingredient being a simple, but effective argument of Ab\'ert and
Nikolov~\cite[Prop.~13]{AbNi12} (where the condition
$\lvert N \rvert = \infty$ should be added).  The point of view that
an assumption on the increase in the number of generators upon passing
to finite index subgroups (e.g. positive rank gradient) forms a good
framework for proving generalisations of Greenberg's theorem is
already expressed in~\cite{Sh17}.  It remains an interesting problem
to find out which consequences of the more special results mentioned
above hold, in fact, for all pro-$p$ groups of positive rank gradient.
Furthermore, Theorem~\ref{Thm:posRG_normal_sgps_inf_gen} is
reminiscent of a theorem of Gaboriau~\cite[Thm.~6.8]{Ga02} for
countable groups with positive first $\ell^2$-Betti number; possible
connections between these notions and results are discussed in
Section~\ref{subsec:Gaboriau}.  We are grateful to Steffen Kionke for
drawing our attention to Gaboriau's theorem.

Our proof of Theorem~\ref{Thm:posRG_normal_sgps_inf_gen}, indeed our
motivation for proving it, stems from the study of Hausdorff
dimensions of subgroups in pro-$p$ groups of positive rank gradient.
The concept of Hausdorff dimension is commonly employed to measure the
sizes of fractals and subsets of general metric spaces;
see~\cite{Fa14,Ro70}.  In~\cite{BaSh97}, Barnea and Shalev initiated
the use of Hausdorff dimension to study the distribution of closed
subgroups in finitely generated pro-$p$ groups; see also Shalev's
survey article~\cite[\S 4]{Sh00}.  This idea has led to many
interesting applications; compare the references given
in~\cite{KlThZR19}.

\smallskip

Let $G$ be a finitely generated infinite pro-$p$ group with a
filtration $\mathcal{S} \colon G_i$, $i \in \N_0$.  Then,
$\mathcal{S}$ induces a translation-invariant metric on~$G$; the
distance between $x, y\in G$ is
$d_\mathcal{S}(x,y) = \inf \{ \lvert G : G_i \rvert^{-1} \mid x^{-1}y
\in G_i \}$.
It is known that the Hausdorff dimension of a closed subgroup
$H \le_\mathrm{c} G$ with respect to $d_\mathcal{S}$ can be obtained
as a lower limit of logarithmic densities:
\[
\hdim_G^\mathcal{S}(H) = \varliminf_{ i\to \infty} \frac{\log_p \lvert
  HG_i : G_i \rvert}{\log_p \lvert G : G_i \rvert} = \varliminf_{i\to
  \infty} \frac{\log_p \lvert H : H \cap G_i \rvert}{\log_p \lvert G :
  G_i \rvert}.
\]
The \emph{Hausdorff spectrum} of $G$ with respect to $\mathcal{S}$ is
the set
\[
\hspec^\mathcal{S}(G) = \{\hdim_G^\mathcal{S}(H) \mid H \le_\mathrm{c}
G\} \subseteq [0,1].
\]
	
The Hausdorff dimension function and the resulting Hausdorff spectrum
depend significantly on the choice of the filtration~$\mathcal{S}$.
With respect to general filtration series rather unexpected features
can be observed, for a wide range of groups, including the groups
considered in this paper; compare~\cite[Thm.~1.5]{KlThZR19}.
Consequently, one typically concentrates on Hausdorff dimension with
respect to standard group-theoretic filtration series of the pro-$p$
group~$G$, such as: the \emph{Frattini series}~$\mathcal{F}$, the
\emph{Zassenhaus series}~$\mathcal{Z}$ (also known as the
\emph{dimension subgroup} or \emph{Jennings--Lazard--Zassenhaus
  series}), or the \emph{$p$-power series}~$\mathcal{P}$; these are
defined as follows:
\begin{align*}
  \mathcal{F} \colon & \quad \Phi_0(G) = G &&  \text{and} \quad 
                                              \Phi_i(G)=[\Phi_{i-1}(G),\Phi_{i-1}(G)]\Phi_{i-1}(G)^p
                                              \;\;\text{for }  i\geq 1, \\
  \mathcal{Z} \colon & \quad Z_1(G) = G, &&  \text{and} \quad Z_i(G)
                                            = Z_{\lceil
                                            i/p\rceil}(G)^p \prod_{1\leq
                                            j<i}[Z_{j}(G),Z_{i-j}(G)] \;\;
                                            \text{for } i\geq 2, \\
  \mathcal{P} \colon & \quad \Pi_i(G) = G^{p^i} && \hspace*{-.7cm} =
                                                   \langle x^{p^i} \mid x\in G\rangle
                                                   \;\; \text{for }  i\geq 0.
\end{align*}
Another natural choice of filtration, the lower central $p$-series,
leads to somewhat unexpected effects in terms of Hausdorff dimension;
compare~\cite {KlThZR19}.  One of the key features of Hausdorff
dimension is that it agrees with the usual dimension of smooth
submanifolds of euclidean space.  Pleasingly, something similar holds
in the $p$-adic setting: closed subgroups of $p$-adic analytic pro-$p$
groups, being $p$-adic analytic themselves, yield only finitely many
values of Hausdorff dimensions with respect to $\mathcal{F}$,
$\mathcal{Z}$ and~$\mathcal{P}$; see~\cite[Thm.~1.1]{BaSh97} and
\cite[Prop.~1.5]{KlThZR19}.  Whether this property characterises
$p$-adic analytic groups among finitely generated pro-$p$ groups is a
long-standing open problem; see~\cite{KlThZR19} for a partial positive
answer, valid for soluble groups.  While the Hausdorff spectra of
finitely generated non-abelian free pro-$p$ groups were perhaps
suspected to be quite large, no proof was found to show that they are
merely infinite, with respect to $\mathcal{F}$, $\mathcal{Z}$
or~$\mathcal{P}$, say.

In this paper we confine ourselves -- similar to Shalev~\cite[\S
4.6]{Sh00} -- to the Frattini series, the Zassenhaus series and
related filtration series.  For the groups we study, the $p$-power
series is beyond current knowledge; even estimating the indices
$\lvert F : F^{p^i} \rvert$ of the terms of the series, for finitely
generated non-abelian free pro-$p$ groups~$F$, amounts to estimating
the corresponding Burnside numbers which are largely unknown.
		
Our main results solve problems highlighted in~\cite{BaSh97}
and~\cite[\S 4.6]{Sh00}; in particular, they finally remove
non-abelian free pro-$p$ groups from the list of potential
counter-examples to the desired characterisation of $p$-adic analytic
pro-$p$ groups $G$ by the finiteness of their Hausdorff spectra, with
respect to $\mathcal{F}$ or~$\mathcal{Z}$.

For the Frattini series $\mathcal{F}$, we can, in fact, remove a much
larger class of potential counter-examples, by virtue of the
techniques used to prove Theorem~\ref{Thm:posRG_normal_sgps_inf_gen}.

\begin{Thm}\label{Thm:full-spectrum-Frattini}
  Let $G = H_1 \times \ldots \times H_r$ be a non-trivial finite direct
  product of finitely generated pro-$p$ groups $H_j$ of positive rank
  gradient.  Then $G$ has full Hausdorff spectrum
  $\hspec^{\mathcal{F}}(G) = [0,1]$ with respect to the Frattini
  series~$\mathcal{F}$.
\end{Thm}

For the Zassenhaus series~$\mathcal{Z}$, we rely on more
intricate Lie-theoretic methods, which do not readily extend from
free groups to general groups of positive rank gradient.
In order to translate between certain subgroups of a
given free pro-$p$ group $F$ and Lie subalgebras of the associated
free restricted Lie algebra, we make use of power-commutator
factorisations in~$F$, which result from Hall's commutator
collection process.

\begin{Thm}\label{thm:full-spectrum-Zassenhaus-free}
  Let $F$ be a finitely generated non-abelian free pro-$p$ group.
  Then $F$ has full Hausdorff spectrum
  $\hspec^{\mathcal{Z}}(F) = [0,1]$ with respect to the Zassenhaus
  series~$\mathcal{Z}$.
\end{Thm}

It is worth comparing Theorem~\ref{thm:full-spectrum-Zassenhaus-free} to
what was previously known about the Hausdorff spectrum
$\hspec^\mathcal{Z}(F)$ of a non-abelian free pro-$p$
group~$F$; see~\cite[\S 4.6]{Sh00}.  Lie-theoretic arguments can be
used to deduce with relative ease that non-trivial normal subgroups of
$F$ have Hausdorff dimension $1$ with respect to~$\mathcal{Z}$.  In
addition, Shalev~\cite[Thm.~4.10]{Sh00} sketched a high-level argument
to locate finitely many rational values of the form $e^{-1}$, for
certain $e \in \N$, in the Hausdorff spectrum $\hspec^\mathcal{Z}(F)$,
but stressed at the time that it was ``unclear how to construct
subgroups of $F$ whose Hausdorff dimension is irrational.''

\smallskip

In addition to free groups, we also consider Demushkin
pro-$p$ groups.  They are, by definition, Poincar\'e duality
pro-$p$ groups of dimension~$2$ and thus can be regarded as pro-$p$
analogues of surface groups; indeed, the pro-$p$ completions of
orientable surface groups are Demushkin groups.  In this way they are
rather close to  finitely generated free pro-$p$ groups.
Demushkin groups also occur naturally as Galois groups in
  algebraic number theory, as mentioned above.  We refer
to~\cite[Sec.~12.3] {Wi98} for the classification of Demushkin groups
of depth $q \neq 2$ and further details; Demushkin groups of depth
$q=2$ are treated in~\cite{La67,Se95}.

The next result applies to mixed finite direct products
  of free and Demushkin pro-$p$ groups, equipped with an iterated
  verbal filtration, and can be established by adapting our proof of
Theorem~\ref{Thm:full-spectrum-Frattini}.

\begin{Defi*}
  A filtration $\mathcal{S} \colon G_i$, $i \in \N_0$, of~$G$ such
  that $G_i$ is a verbal subgroup of $G_{i-1}$ for each~$i \in \N$ is
  called an \emph{iterated verbal filtration}.
\end{Defi*}

Typical examples of iterated verbal filtrations are the Frattini
filtration and the \emph{iterated $p$-power series} defined by
\[
\mathcal{P}^*:\quad \Pi_0^*(G) = G, \quad \text{and} \quad \Pi_i^*(G)
= \Pi_{i-1}^*(G)^p = \langle x^p \mid x \in \Pi_{i-1}^*(G)
\rangle\quad \text{for $i\geq 1$.}
\]

\begin{Thm} \label{thm:full_spec_iterated_verbal} Let
  $G = E_1 \times \ldots \times E_r$ be a finite direct product of
  finitely generated pro-$p$ groups $E_j$, each of which is either
  free or Demushkin.  Suppose that at least one factor is non-soluble.
  Let $\mathcal{S} \colon G_i$, $i \in \N_0$, be an iterated verbal
  filtration of~$G$.  Then $G$ has full Hausdorff spectrum
  $\hspec^{\mathcal{S}}(G) = [0,1]$.
\end{Thm}

In a similar direction, we extend
Theorem~\ref{thm:full-spectrum-Zassenhaus-free} to non-soluble
Demushkin pro-$p$ groups and derive as a common generalisation a
corresponding result for mixed finite direct products of free and
Demushkin pro-$p$ groups.

\begin{Thm}\label{thm:full-spectrum-Zassenhaus-Demushkin}
  Let $D$ be a non-soluble Demushkin pro-$p$ group.  Then $D$ has
  full Hausdorff spectrum $\hspec^{\mathcal{Z}}(D) = [0,1]$ with
  respect to the Zassenhaus series~$\mathcal{Z}$.
\end{Thm}

\begin{Cor} \label{cor:full-spectrum-Zassenhaus-mixed} Let
  $G = E_1 \times \ldots \times E_r$ be a finite direct product of
  finitely generated pro-$p$ groups $E_j$, each of which is either
  free or Demushkin.  Suppose that at least one factor is non-soluble.
  Then $G$ has full Hausdorff spectrum
  $\hspec^{\mathcal{Z}}(G) = [0,1]$ with respect to the Zassenhaus
  series~$\mathcal{Z}$.
\end{Cor}

The study of Hausdorff spectra of pro-$p$ right-angled
Artin groups -- defined via presentations analogously to the
discrete case -- provides a natural avenue for future research.
Above we already formulated our results to cover the basic case of
direct products of free pro-$p$ groups and, indeed, more general
direct products.  As a supplement, we can also describe the Hausdorff
dimensions of normal subgroups in such direct products.

\smallskip

The \emph{normal Hausdorff spectrum} of a finitely generated pro-$p$
group~$G$, with respect to a filtration~$\mathcal{S}$, is
\[
\hspec_\trianglelefteq^\mathcal{S}(G) = \{\hdim_G^\mathcal{S}(H) \mid
H \trianglelefteq_\mathrm{c} G\};
\]
compare~\cite[\S 4.7]{Sh00} and \cite{KlThxx}.

\begin{Thm}\label{thm:direct_prods}
  Let $G = H_1 \times \ldots \times H_r$ be a non-trivial finite
  direct product of finitely generated pro-$p$ groups $H_j$ of
  positive rank gradient.  For $1 \le j \le r$, set
  $\alpha_j = \hdim_G^\mathcal{F}(H_j)$, with respect to the Frattini
  series~$\mathcal{F}$.  Then $G$ has normal Hausdorff spectrum
  $\hspec_\trianglelefteq^{\mathcal{F}}(G) = \{ \sum_{j \in J}
  \alpha_j \mid J \subseteq \{1,\ldots,r\} \}$
  with respect to~$\mathcal{F}$.
\end{Thm}

\begin{Thm}\label{thm:direct_prods_free_Dem}
  Let $G = E_1 \times \ldots \times E_r$ be a finite direct product of
  finitely generated pro-$p$ groups $E_j$, each of which is either
  free or Demushkin.  Suppose that at least one factor is non-soluble.
  For $1 \le j \le r$, set $\alpha_j = \hdim_G^\mathcal{S}(E_j)$, with
  respect to an iterated verbal filtration~$\mathcal{S}$ of~$G$.  Then
  $G$ has normal Hausdorff spectrum
  $\hspec_\trianglelefteq^{\mathcal{S}}(G) = \{ \sum_{j \in J}
  \alpha_j \mid J \subseteq \{1,\ldots,r\} \}$
  with respect to~$\mathcal{S}$.
\end{Thm}

In addition to these general theorems we have the
following concrete results for products of free pro-$p$ groups.

\begin{Thm}\label{thm:direct_prods_free}
  Let $G = F_1 \times \ldots \times F_r$ be a non-trivial finite
  direct product of finitely generated free pro-$p$ groups~$F_j$.  Put
  \[
  d = \max \{ d(F_j) \mid 1 \le j \le r \} \qquad \text{and} \qquad
  t = \lvert \{ j \mid 1 \le j \le r, \, d(F_j) = d \}
  \rvert.
  \]

  Let $\mathcal{S}$ be an iterated verbal filtration of~$G$ and let
  $\mathcal{Z}$ be the Zassenhaus series of~$G$.  Then $G$ has the normal
  Hausdorff spectra
  \[
  \hspec_\trianglelefteq^{\mathcal{S}}(G) =
  \hspec_\trianglelefteq^{\mathcal{Z}}(G) = \{0, \nicefrac{1}{t},
  \ldots, \nicefrac{(t-1)}{t}, 1 \}.
  \]
\end{Thm}

We conclude the introduction with a brief outline of the general
strategy and the organisation of the paper, following a suggestion of
the referee.

\medskip

\noindent \textit{General Strategy and Organisation.}
Theorem~\ref{Thm:posRG_normal_sgps_inf_gen}, which is proved in
Section~\ref{sec:frattini}, arises from the study of Hausdorff
dimensions of subgroups of pro-$p$ groups $G$ of positive rank
gradient.  Its proof is based on the observation that finitely
generated, infinite-index subgroups $H \le_\mathrm{c} G$ have
Hausdorff dimension $0$, whereas infinite subgroups
$K \le_\mathrm{c} G$ that are subnormal via finitely generated
successive quotients have Hausdorff dimension~$1$.

The main idea underlying the proofs of
Theorems~\ref{Thm:full-spectrum-Frattini},
\ref{thm:full-spectrum-Zassenhaus-free}
and~\ref{thm:full_spec_iterated_verbal} is to exploit suitable
variations of the following result.

\begin{Thm-no-nr}[{\cite[Thm.~5.4]{KlThZR19}}]
  Let $G$ be a countably based pro-$p$ group with a filtration series
  $\mathcal{S}$.  Let $H \le_{\mathrm{c}} G$ be a closed subgroup such
  that $\xi =\hdim_G^\mathcal{S}(H) > 0$ is given by a proper limit.
  Suppose further that $\eta\in[0,\xi) $ is such that every finitely
  generated closed subgroup $K \le_{\mathrm{c}} H$ satisfies
  $\hdim_G^\mathcal{S}(K)\le \eta$. Then
  $(\eta,\xi]\subseteq \hspec^{\mathcal{S}}(G)$.
\end{Thm-no-nr}

The theorem can be applied directly to pro-$p$ groups $G$ of positive
rank gradient, equipped with the Frattini series~$\mathcal{F}$.  For
this, one takes $H \trianglelefteq_\mathrm{c} G$ such that
$\lvert H \rvert$ and $\lvert G : H \rvert$ are both infinite; then
$\xi=1$ and $\eta=0$ yield $\hspec^\mathcal{F}(G) = [0,1]$.
Proposition~\ref{prop:KlThZR19} provides a technical modification of
the theorem, tailor-made for dealing with direct products; this
approach leads to a proof of Theorem~\ref{Thm:full-spectrum-Frattini}
in Section~\ref{sec:frattini}.
 
For non-abelian free pro-$p$ groups and for non-soluble Demushkin
pro-$p$ groups, we generalise in Section~\ref{sec:verbal} the relevant
arguments.  This allows us to deal in a similar way with Hausdorff
dimension with respect to arbitrary iterated verbal filtrations, and
we obtain Theorem~\ref{thm:full_spec_iterated_verbal}.

In Section~\ref{sec:Zassenhaus}, we focus on non-abelian free pro-$p$
groups~$F$, equipped with the Zassenhaus filtration~$\mathcal{Z}$.
A standard construction yields a (graded) free restricted
$\F_p$-Lie algebra~$\mathsf{R} = \mathsf{R}(F)$, associated to $F$
via~$\mathcal{Z}$.  Moreover, the Hausdorff dimension function with
respect to~$\mathcal{Z}$ has a natural Lie-theoretic counterpart, the
so-called density function. \emph{A priori } it is not
 clear that relevant density results for restricted Lie subalgebras
 of~$\mathsf{R}$ can be `pulled back' to Hausdorff dimension results
 for subgroups of~$F$, but, since $F$ is free, the fact that finitely
 generated subalgebras of infinite codimension in~$\mathsf{R}$ have
 density~$0$ (see
 Proposition~\ref{prop:finitely_gen_subalg_density_0}) implies that
 finitely generated infinite-index subgroups of $F$ have Hausdorff
  dimension~$0$ with respect to~$\mathcal{Z}$ %
  (see Corollary~\ref{cor:fg-H-gives-fg-Lie}).
    Using the Interval Theorem stated above, we obtain
  Theorem~\ref{thm:full-spectrum-Zassenhaus-free}.  In addition we
  derive Theorem~\ref{thm:all_densities_freelie} and
  Corollary~\ref{cor:all_densities_restricted} which show, in
  particular, that free $\F_p$-Lie algebras and free restricted
  $\F_p$-Lie algebras have full density spectra.

Much less seems to be known about the restricted Lie algebras
associated to pro-$p$ groups that resemble free pro-$p$ groups.  For
instance, Lie-theoretic analogues of rank gradient seem to have
received little attention. Even extending
  Theorem~\ref{thm:full-spectrum-Zassenhaus-free} to Demushkin pro-$p$
  groups turns out to be somewhat subtle.  By analysing relevant
  Hilbert--Poincar{\'e} series, we obtain
  Theorem~\ref{thm:full-spectrum-Zassenhaus-Demushkin} in
  Section~\ref{sec:Demushkin}.  With some extra care we subsequently
  derive Corollary~\ref{cor:full-spectrum-Zassenhaus-mixed}.

Theorems~\ref{thm:direct_prods}, \ref{thm:direct_prods_free_Dem}
and~\ref {thm:direct_prods_free} which deal with Hausdorff dimensions
of normal subgroups are proved in Section~\ref{sec:normal-spectra}.

\medskip

\noindent \textit{Acknowledgement.} We thank Andrei
  Jaikin-Zapirain for general comments and particularly for alerting
  us to results in \cite{ErJZ13} that led to a simplification of our
  proof of Theorem~\ref{thm:full-spectrum-Zassenhaus-free}.
  Furthermore, his comments encouraged us to extend our result to
  Demushkin groups in
  Theorem~\ref{thm:full-spectrum-Zassenhaus-Demushkin}.
  The third author thanks Ilir Snopche for several discussions about
  Demushkin groups.


\section{Groups with positive rank gradient and the Frattini
  series}\label{sec:frattini}

In this section we establish
Theorems~\ref{Thm:posRG_normal_sgps_inf_gen}
and~\ref{Thm:full-spectrum-Frattini}, working with Hausdorff
dimensions with respect to the Frattini series.  We also present an
alternative and more direct proof of
Theorem~\ref{Thm:posRG_normal_sgps_inf_gen}, which was later suggested
to us by Mark Shusterman.  It is based on an idea that is used in the
proofs of~\cite[Prop.~13]{AbNi12} and~\cite[Thm.~2.7]{Sh17}.
	
Since the Frattini filtration encodes subgroup generation, it is
naturally linked to rank gradient.  Other filtrations encode different
properties of subgroups, linked to as yet unexplored group invariants,
analogous to rank gradient.  Our proof of
Theorem~\ref{Thm:posRG_normal_sgps_inf_gen} in the language of
Hausdorff dimension allows for generalisations in this direction.

  
\subsection{Finitely generated subgroups}
We say that a closed subgroup $H$ of a finitely generated pro-$p$
group $G$ has \emph{strong Hausdorff dimension} with respect to a
filtration $\mathcal{S} \colon G_i$, $i \in \N_0$, if
\[
\hdim_G^{\mathcal{S}}(H) = \lim _{i \to \infty} \frac{\log_p \lvert
  HG_i : G_i \rvert}{\log_p \lvert G : G_i \rvert}
\]
is given by a proper limit.  We begin with an elementary, but central
lemma.
	
\begin{Lem}\label{Lem:ratio_gens_decreases}
  Let $G$ be a finitely generated pro-$p$ group of positive rank
  gradient, with a filtration $\mathcal{S} \colon G_i$, $i \in \N_0$.
  \begin{enumerate}[\rm (1)]
  \item If $H \le_\mathrm{c} G$ is finitely generated and
    $\lvert G : H \rvert = \infty$, then $\lim\limits_{i \to \infty}
    \frac{\dd(H\cap G_i)}{\dd(G_i)} =  0$.
  \item If $H \le_\mathrm{c} G$ is infinite, then
    $\lim\limits_{i \to \infty} \frac{\dd(HG_i)}{\dd(G_i)} = 0$.
  \end{enumerate}
\end{Lem}
	
\begin{proof}
  Put $R = \RG(G) = \RG(G,\mathcal{S}) > 0$, and note that
  $\dd(G_i) \geq R\, \lvert G : G_i \rvert$ for all $i \in \N_0$.

  First suppose that $H \le_\mathrm{c} G$ is finitely generated and
  $\lvert G : H \rvert = \infty$. Then each $\dd(H\cap G_i)$ can be
  bounded by means of the Schreier index formula applied to
  $H\cap G_i \le_\mathrm{o} H$, and $\lvert G : HG_i \rvert$ tends to
  infinity as $i$ does.  This yields
  \[
  \frac{\dd(H\cap G_i)}{\dd(G_i)} \leq \frac{(\dd(H)-1) \lvert H : H
    \cap G_i \rvert}{R\, \lvert G : G_i \rvert} \leq \frac{\dd(H) \lvert HG_i : G_i
    \rvert}{R\, \lvert G : HG_i \rvert \lvert HG_i : G_i \rvert} =
  \frac{\dd(H)}{R\, \lvert G : HG_i \rvert} \xrightarrow[i\to \infty]{} 0.
  \]
		
  Now suppose that $H \le_\mathrm{c} G$ is infinite.  Then, similarly,
  each $\dd(HG_i)$ can be
  bounded by means of the Schreier index formula applied to
  $HG_i \le_\mathrm{o} G$, and $\lvert HG_i : G_i \rvert$ tends to
  infinity as $i$ does.  This yields
  \[
  \frac{\dd(HG_i)}{\dd(G_i)} \leq \frac{(\dd(G)-1) \lvert G : H G_i
    \rvert}{R\, \lvert G : G_i \rvert} \leq \frac{\dd(G)}{R\, \lvert
    HG_i : G_i \rvert} \xrightarrow[i\to\infty]{} 0. \qedhere
  \]
\end{proof}

We require two further auxiliary results, which are closely related to
one another.
	
\begin{Prop} \label{Prop:fg_subgroups_hdim_0} Let $G$ be a finitely
  generated pro-$p$ group of positive rank gradient, and let
  $H \le_\mathrm{c} G$ be finitely generated with
  $\lvert G : H \rvert = \infty$.  Then $H$ has strong Hausdorff
  dimension $\hdim_G^\mathcal{F}(H) = 0$ with respect to the Frattini
  series~$\mathcal{F}$ of~$G$.
\end{Prop}

\begin{proof}
  We show that $\hdim_G^\mathcal{F}(H) \le \nicefrac{1}{n}$ for any
  given~$n \in \N$, and from the argument it will be clear that $H$
  has strong Hausdorff dimension.

  We write $\mathcal{F} \colon G_i = \Phi_i(G)$, $i \in \N_0$.
  Clearly, $\log_p \lvert G_i : G_{i+1} \rvert = \dd(G_i)$ for
  $i \ge 0$, and from
  $\Phi( H \cap G_i) \subseteq H\cap \Phi(G_i) = H\cap G_{i+1}$ we
  deduce that
  \[
  \log_p \lvert H \cap G_i : H \cap G_{i+1} \rvert \leq \log_p \lvert
  H \cap G_i : \Phi(H \cap G_i) \rvert = \dd(H \cap G_i) \quad \text{for $i\ge 0$.}
  \]
  By Lemma~\ref{Lem:ratio_gens_decreases}, we find $i_0 \in \N$ such
  that $\dd(H \cap G_i)/\dd(G_i) \leq \nicefrac{1}{n}$ for
  $i \geq i_0$.  These observations yield, for $i \ge i_0$,
  \[
  \frac{\log_p \lvert H \cap G_{i_0} : H \cap G_i \rvert}{\log_p
    \lvert G_{i_0} : G_i \rvert} \leq \frac{\sum_{j = i_0}^{i-1}
    \dd(H\cap G_j)}{\sum_{j = i_0}^{i-1}\dd(G_j)} \leq \frac{\sum_{j =
      i_0}^{i-1} \dd(H\cap G_j)}{\sum_{j = i_0}^{i-1} n\dd(H\cap G_j)}
  = \frac{1}{n},
  \]
  and hence
  \[
  \hdim_G^\mathcal{F}(H) = \varliminf_{\substack{i\to \infty \\[2pt] i
      \ge i_0}} \frac{\overbrace{\log_p \lvert H : H \cap G_{i_0}
      \rvert}^{\text{constant}} + \log_p \lvert H \cap G_{i_0} : H
    \cap G_i \rvert}{\underbrace{\log_p \lvert G : G_{i_0}
      \rvert}_{\text{constant}} + \log_p \lvert G_{i_0} : G_i \rvert}
  \le \frac{1}{n}. \qedhere
  \]
\end{proof}
	
\begin{Prop} \label{Prop:normal_subgroups_hdim_1} Let $G$ be a
  finitely generated pro-$p$ group of positive rank gradient, and let
  $N \trianglelefteq_\mathrm{c} G$ be an infinite normal subgroup.
  Then $N$ has strong Hausdorff dimension $\hdim_G^\mathcal{F}(N) = 1$
  with respect to the Frattini series~$\mathcal{F}$ of~$G$.
\end{Prop}

\begin{proof}
  We write $\mathcal{F} \colon G_i = \Phi_i(G)$, $i \in \N_0$, and
  observe that
  \[
  \hdim_G^\mathcal{F}(N) = \varliminf_{i\to \infty} \frac{\log_p \lvert NG_i : G_i
    \rvert}{\log_p \lvert G : G_i \rvert} = 1 - \varlimsup_{i\to
    \infty} \frac{\log_p \lvert G : NG_i \rvert}{\log_p \lvert G : G_i
    \rvert}.
  \]
  Hence it suffices to show that the upper limit on the right-hand
  side equals~$0$ and the Hausdorff dimension will automatically be
  strong.  Proceeding in a similar way to the proof of
  Proposition~\ref{Prop:fg_subgroups_hdim_0}, we show that the upper
  limit is at most $\nicefrac{1}{n}$ for any given~$n \in \N$.

  Again, $\log_p \lvert G_i : G_{i+1} \rvert = \dd(G_i)$ for
  $i \ge 0$, and we observe that 
  \[
  \log_p \lvert N G_i : N G_{i+1} \rvert = \log_p \lvert G_iN/N :
  \Phi(G_iN/N) \rvert = \dd(G_iN/N) \le \dd(NG_i) \quad \text{for
    $i\ge 0$.}
  \]
  By Lemma~\ref{Lem:ratio_gens_decreases}, we find $i_0 \in \N$ such
  that $\dd(N G_i)/\dd(G_i) \leq \nicefrac{1}{n}$ for $i \geq i_0$.
  These observations yield, for $i \ge i_0$,
  \[
  \frac{\log_p \lvert NG_{i_0} : NG_i \rvert}{\log_p \lvert G_{i_0} :
    G_i \rvert} \leq \frac{\sum_{j = i_0}^{i-1} \dd(N G_j)}{\sum_{j =
      i_0}^{i-1}\dd(G_j)} \leq \frac{\sum_{j = i_0}^{i-1} \dd(N
    G_j)}{\sum_{j = i_0}^{i-1} n\dd(N G_j)} = \frac{1}{n},
  \]
  and hence
  \[
  \varlimsup_{i\to \infty} \frac{\log_p \lvert G : NG_i \rvert}{\log_p
    \lvert G : G_i \rvert} = \varlimsup_{\substack{i\to \infty \\[2pt]
      i \ge i_0}} \frac{\overbrace{\log_p \lvert G : N G_{i_0}
      \rvert}^{\text{constant}} + \log_p \lvert N G_{i_0} : N G_i
    \rvert}{\underbrace{\log_p \lvert G : G_{i_0}
      \rvert}_{\text{constant}} + \log_p \lvert G_{i_0} : G_i \rvert}
  \le \frac{1}{n}. \qedhere
  \]
\end{proof}

From Proposition~\ref{Prop:fg_subgroups_hdim_0} we deduce the
following corollary.

\begin{Cor} \label{cor:fg_quotient_same_hdim} Let $G$ be a finitely
  generated pro-$p$ group of positive rank gradient, and let
  $\mathcal{F}$ denote the Frattini series of~$G$.  Then the following
  hold.
  \begin{enumerate}[\rm (1)]
  \item If $A, B \le_\mathrm{c} G$ are such that
    $AB = \{ ab \mid a \in A, b \in B \} \le_\mathrm{c} G$ is a
    subgroup and $A$ is finitely generated with
    $\lvert G : A \rvert = \infty$, then
    $\hdim_G^{\mathcal{F}}(AB) = \hdim^{\mathcal{F}}_G(B)$.
  \item If $K \trianglelefteq_\mathrm{c} H \le_\mathrm{c} G$
    such that $H/K$ is finitely generated and
    $\lvert G : H \rvert = \infty$, then
    $\hdim^{\mathcal{F}}_G(H) = \hdim^{\mathcal{F}}_G(K)$.
  \end{enumerate}
\end{Cor}

\begin{proof}
  (1) Let $A, B \le_\mathrm{c} G$ be such that
  $AB = \{ ab \mid a \in A, b \in B \} \le G$ and $A$ is finitely
  generated.  Using Proposition~\ref{Prop:fg_subgroups_hdim_0}, we
  obtain
  \begin{align*}
    \hdim^{\mathcal{F}}_G(AB)
    & = \varliminf_{i\to \infty} \frac{\log_p \lvert ABG_i : B
      G_i \rvert  +\log_p \lvert B G_i : G_i \rvert}{\log_p \lvert
      G : G_i \rvert}\\
    &\leq \varliminf_{i\to \infty} \frac{\log_p \lvert AG_i : 
      G_i \rvert  +\log_p \lvert B G_i : G_i \rvert}{\log_p \lvert
      G : G_i \rvert}\\
    & =  \lim_{i\to \infty} \frac{\log_p \vert AG_i:G_i
      \rvert}{\log_p \lvert G:G_i \rvert} + \varliminf_{i\to
      \infty} \frac{\log_p \lvert BG_i:G_i \rvert}{\log_p \lvert
      G:G_i \rvert} \\
    & = \hdim^{\mathcal{F}}_G(B). 
  \end{align*}
	
  (2) The claim follows from (1), upon taking $B = K$ and
  $A = \langle a_1, \ldots, a_d \rangle$ such that $AB = H$.
\end{proof}

We are ready to prove our first result,
Theorem~\ref{Thm:posRG_normal_sgps_inf_gen}.
         
\begin{proof}[Proof of Theorem~\ref{Thm:posRG_normal_sgps_inf_gen}]
  The implications $(3) \Rightarrow (2) \Rightarrow (1)$ are
  straightforward, and it remains to show $(1) \Rightarrow (3)$.
  Suppose that $H$ contains an infinite subgroup $K \le_\mathrm{c} G$
  that is subnormal via finitely generated successive quotients.  This
  means: $K \le_\mathrm{c} H$ with $\lvert K \rvert = \infty$ and
  there is a subnormal chain
  $G = K_0\trianglerighteq K_1\trianglerighteq \ldots \trianglerighteq
  K_s = K$
  such that each factor $K_{i-1}/K_i$, for $i \in \{1,\ldots,s\}$, is
  finitely generated.  By Proposition~\ref{Prop:fg_subgroups_hdim_0}
  it suffices to show that $\hdim_G^\mathcal{F}(K) = 1$, where
  $\mathcal{F}$ denotes the Frattini series of~$G$.

  Choose $j \in \{0,1,\ldots,s\}$ maximal with respect to
  $\lvert G : K_j \rvert < \infty$.  Clearly, if $j = s$ then
  $K \le_\mathrm{o} G$ and hence $\hdim_G^\mathcal{F}(K) = 1$.  Now
  suppose that $j < s$.  As $K_j \le_\mathrm{o} G$, we find $i \in \N$
  such that $\Phi_i(G)$, the $i$th term of the Frattini series, is
  contained in~$K_j$.  Replacing $G$ by $\Phi_i(G)$ and all other
  subgroups by their intersections with $\Phi_i(G)$, we may assume
  without loss of generality that $G = K_0 = \ldots = K_j$ and 
  $K_{j+1} \trianglelefteq_\mathrm{c} G$.
  Proposition~\ref{Prop:normal_subgroups_hdim_1} implies that
  $\hdim_G^\mathcal{F}(K_{j+1}) =1$.  Now, several applications of 
  Corollary~\ref{cor:fg_quotient_same_hdim}, part~(2), yield that
  $1 = \hdim_G^\mathcal{F}(K_{j+1}) = \hdim_G^\mathcal{F}(K_{j+2}) =
  \ldots = \hdim_G^\mathcal{F}(K_s)$.
\end{proof}

As indicated above, we briefly record an alternative argument that is
similar to the one used in the proofs of ~\cite[Prop.~13]{AbNi12}
and~\cite[Thm.~2.7]{Sh17}.

\begin{proof}[Second proof of Theorem~\ref{Thm:posRG_normal_sgps_inf_gen}]
  As in the first proof, it suffices to show $(1) \Rightarrow (3)$.
  Suppose that $H, K \le_\mathrm{c} G$ and
  $G = K_0\trianglerighteq K_1\trianglerighteq \ldots \trianglerighteq
  K_s = K$
  are given as before.  For a contradiction, assume that
  $\lvert G : H \rvert = \infty$.

  We show that, for $U \trianglelefteq_\mathrm{o} G$,
  \begin{equation} \label{equ:to-0-claim} \frac{\dd(U) - 1}{\lvert G
      :U \rvert} \to 0 \qquad \text{as
      $\lvert G : H U \rvert \to \infty$ and
      $\lvert KU : U \rvert \to \infty$ simultaneously;}
  \end{equation}
  this yields the required contradiction, because $\RG(G) > 0$ and
  $\lvert G : H \rvert, \lvert K \rvert = \infty$.

  Observe that
  \begin{equation} \label{equ:d-U-estimate} \dd(U) \le \dd(U \cap H) +
    \sum\nolimits_{i = 1}^s \dd((U \cap K_{i-1})/(U \cap K_i)).
  \end{equation}
  From
  $\dd(U \cap H) \le (\dd(H) - 1) \lvert H : U \cap H \rvert + 1 \le
  \dd(H) \lvert HU : U \rvert + 1$ we deduce that
  \begin{equation} \label{equ:U-cap-H-to-0} \frac{\dd(U \cap H)
      -1}{\lvert G : U \rvert} \le \frac{\dd(H)}{\lvert G : HU
      \rvert} \to 0 \qquad \text{as
      $\lvert G : HU \rvert \to \infty$.}
  \end{equation}
  For $1 \le i \le s$, we observe that
  $(U \cap K_{i-1})/(U \cap K_i) \cong (U K_i \cap K_{i-1})/K_i$ and
  that
  $\lvert K_{i-1} : U K_i \cap K_{i-1} \rvert = \lvert K_{i-1} U : K_i
  U \rvert$; this gives
  \[
  \dd((U \cap K_{i-1})/(U \cap K_i) \le (\dd(K_{i-1}/K_i) - 1) \lvert
  K_{i-1} U : K_i U \rvert \le \dd(K_{i-1}/K_i) \lvert G : K U \rvert.
  \]
  We conclude that
  \begin{equation} \label{equ:rest-to-0} \frac{\sum_{i = 1}^s \dd((U
      \cap K_{i-1})/(U \cap K_i))}{\lvert G : U \rvert} \le
    \frac{\sum_{i = 1}^s \dd(K_{i-1}/K_i)}{\lvert K U : U \rvert} \to
    0 \qquad \text{as $\lvert KU : U \rvert \to \infty$.}
  \end{equation}
  The claim~\eqref{equ:to-0-claim} follows
  from~\eqref{equ:d-U-estimate}, \eqref{equ:U-cap-H-to-0}
  and~\eqref{equ:rest-to-0}.
\end{proof}


\subsection{Rank gradient and $\ell^2$-Betti
  numbers} \label{subsec:Gaboriau}

Theorem~\ref{Thm:posRG_normal_sgps_inf_gen} can be regarded as a
pro-$p$ analogue of Gaboriau's theorem~\cite[Thm.~6.8]{Ga02} for
countably infinite groups with positive first $\ell^2$-Betti number.
	
Several results in the literature point to a relationship between the
rank gradient and the first $\ell^2$-Betti number of finitely
generated groups.  Let $\Gamma$ be a finitely generated residually
finite group and denote its first $\ell^2$-Betti number
by~$b_1^{(2)}(\Gamma)$.  By combining \cite[Thm.~1]{AbNi12} and
\cite[Cor.~3.23]{Ga02}, it is established that
\[
\RG(\Gamma, \mathcal{S}) = \lim_{i \to \infty}
\frac{\dd(\Gamma_i)-1}{\lvert G : \Gamma_i \rvert} \geq
b_1^{(2)}(\Gamma),
\]
for any descending chain $\mathcal{S} \colon \Gamma_i$, $i \in \N_0$,
of finite-index normal subgroups
$\Gamma_i \trianglelefteq_\mathrm{f} \Gamma$ with trivial intersection
$\bigcap_i \Gamma_i =1$.  In fact, in all known examples, the two
quantities coincide, but it is not known whether this is true in
general.

In order to provide a link between $\Gamma$ and the rank gradient of
its pro-$p$ completion, it is useful to introduce the (relative)
\emph{mod-$p$ homology gradient} (or \emph{$p$-rank gradient}) of
$\Gamma$ with respect to a descending chain
$\mathcal{S} \colon \Gamma_i$, $i \in \N_0$, consisting of
finite-$p$-power-index normal subgroups; it is
\[
\RG_p(\Gamma,\mathcal{S}) = \lim_{i\to \infty} \frac{\dim_{\F_p}
  \mathrm{H}_1(\Gamma_i, \F_p) }{\lvert \Gamma : \Gamma_i \rvert}.
\]
Notice that, if $\bigcap_i \Gamma_i =1$ and the completion
$\widehat{\Gamma}_\mathcal{S}$ with respect to $\mathcal{S}$ and the
pro-$p$ completion~$\widehat{\Gamma}_p$ of $\Gamma$ coincide, then
the (relative) mod-$p$ homology gradient coincides with the rank
gradient of~$\widehat{\Gamma}_p$.  In general, it is not known under
what conditions the equality
$\RG_p(\Gamma, \mathcal{S}) =\RG(\widehat{\Gamma}_p)$ holds;
compare~\cite{AbNi12}.
We thank Andrei Jaikin-Zapirain for pointing this out to us.

Every group $\Gamma$ with positive mod-$p$ homology gradient has
exponential subgroup growth and, if $\Gamma$ is finitely presented, it
virtually maps onto a non-abelian free group; see \cite{La09} and
\cite{La10}, respectively.  Inspired by L\"uck's approximation theorem
for finitely presented residually finite groups, Ershov and
L\"uck~\cite{ErLu14} showed that
$\RG_p(\Gamma,\mathcal{S}) \geq b_1^{(2)}(\Gamma)$ for every finitely
generated infinite group $\Gamma$, equipped with a descending chain
$\mathcal{S}$ of finite-$p$-power-index normal subgroups that have
trivial intersection.  In particular, for every residually finite-$p$
group $\Gamma$ with $b_1^{(2)}(\Gamma)>0$, the pro-$p$ completion
$\widehat{\Gamma}_p$ satisfies the hypotheses of
Theorem~\ref{Thm:posRG_normal_sgps_inf_gen}.

It would be of considerable interest to prove a version of
Theorem~\ref{Thm:posRG_normal_sgps_inf_gen}, where $G$ is taken to be
finitely generated (discrete), residually finite and of positive rank
gradient. As this would be an analogue of Gaboriau's theorem for
discrete groups of positive rank gradient, it might explain, or at
least strengthen, the connection between rank gradient and first
$\ell^2$-Betti number.  A result like this is beyond the scope of the
present paper.  Nevertheless, without any extra work, we can state at
least a partial result in this direction.

\begin{Cor} \label{Cor:res_p_posRG_Gaboriau}
  Let $\Gamma$ be a finitely generated, residually finite-$p$ group of
  positive mod\nobreakdash-$p$ homology gradient, with pro-$p$
  completion~$\widehat{\Gamma}_p$.  Suppose further:
  \begin{itemize}
  \item[$(\dagger)$] For every finitely generated subgroup
    $\Delta \le \Gamma$ with $\lvert \Gamma : \Delta \rvert = \infty$
    the closure $\overline{\Delta}$ in $\widehat{\Gamma}_p$ satisfies
    $\lvert \widehat{\Gamma}_p : \overline{\Delta} \rvert = \infty$.
  \end{itemize}
  Then every finitely generated subgroup of $\Gamma$ that is subnormal
  via finitely generated successive quotients is either finite or of
  finite index in~$\Gamma$.
\end{Cor}

\begin{proof}
  Let $\Delta\leq \Gamma$ be subnormal via finitely generated
  successive quotients and assume for a contradiction that $\Delta$ is
  finitely generated, infinite and of infinite index in $\Gamma$.
  Then, in accordance with our hypothesis,
  $\overline{\Delta} \le_\mathrm{c} \widehat{\Gamma}_p$ has the
  corresponding properties, in contradiction to
  Theorem~\ref{Thm:posRG_normal_sgps_inf_gen}.
\end{proof}
	
\begin{Rmk}
  The condition $(\dagger)$ in
  Corollary~\ref{Cor:res_p_posRG_Gaboriau} is weaker than being
  $p$-WLERF in~\cite{ErJZ13}, but imposes a noticeable
  restriction. For instance, non-abelian-free groups have to be
  excluded: it was shown in~\cite{Ba04} that, for any $m \in \N_0$,
  the subgroup $\Delta = \langle x[y,x], y, z_1, \ldots, z_m\rangle$
  of the free group $\Psi$ on $m+2$ elements $x, y, z_1,\ldots,z_m$ is
  dense in the pro-nilpotent topology of $\Psi$ (and thus, in the
  pro-$p$ topology for any~$p$) but of infinite index in $\Psi$ (and,
  in fact, isolated).  We thank Michal Ferov for alerting us to this
  example.

  A slight modification shows that any group $\Gamma$ that is a free
  product of a virtually non-abelian-free group and any other finitely
  generated group does not satisfy $(\dagger)$.
\end{Rmk}

\begin{Qu}
  Identify examples of finitely generated, residually finite-$p$
  groups of positive mod-$p$ homology gradient that satisfy condition
  $(\dagger)$ in Corollary~\ref{Cor:res_p_posRG_Gaboriau}.
\end{Qu}

Some potential candidates include the examples constructed
in~\cite{Os11} and~\cite{SP12}. 

\begin{Rmk}
  The arguments in the second proof of
  Theorem~\ref{Thm:posRG_normal_sgps_inf_gen} can be used verbatim to
  prove an analogue of Corollary~\ref{Cor:res_p_posRG_Gaboriau} for
  finitely generated, residually finite groups $\Gamma$ of positive
  rank gradient, where the pro-$p$ completion $\widehat{\Gamma}_p$ is
  replaced by the profinite completion~$\widehat{\Gamma}$.  We do not
  pursue this line of investigation here, but merely point out that
  the analogue of $(\dagger)$ in this setting is a weaker assumption,
  since free discrete groups do satisfy it;
  compare~\cite[Thm.~5.1]{Ha49}.
\end{Rmk}

We end this section with another natural question.  It appears that
all currently known examples of finitely generated pro-$p$ groups of
positive rank gradient have positive $p$-deficiency (which gives a
lower bound for the rank gradient), or satisfy a similar condition.

\begin{Qu}
  Is there a finitely generated pro-$p$ group of positive rank
  gradient whose $p$-deficiency in the sense of~\cite{SP12} (and in
  any similar sense) is~$0$?
\end{Qu}


\subsection{Hausdorff spectrum with respect to the Frattini series}

Consider a direct product $G = A \times B$ of finitely generated
pro-$p$ groups $A$ and~$B$, equipped with a filtration series
$\mathcal{S} \colon G_i$, $i \in \N_0$, and let
$\mathcal{S} \vert_A \colon A_i = A \cap G_i$, $i \in \N_0$, and
$\mathcal{S} \vert_B \colon B_i = B \cap G_i$, $i \in \N_0$, denote
the induced filtration series.  In general, there is no
straightforward reduction of the Hausdorff dimension function
$\hdim_G^\mathcal{S}$ to the related Hausdorff dimension functions
$\hdim_A^\mathcal{S \vert_A}$ and $\hdim_B^\mathcal{S \vert_B}$.  Even
if $G_i = A_i \times B_i$ for all $i \in \N_0$ and if
$H \le_\mathrm{c} G$ decomposes as $H = C \times D$ with
$C \le_\mathrm{c} A$ and $D \le_\mathrm{c} B$, it is \emph{not}
generally true that
$\hdim_G^\mathcal{S}(H) = \hdim_A^{\mathcal{S} \vert_A}(C)
\hdim_G^\mathcal{S}(A) + \hdim_B^{\mathcal{S} \vert_B}(D)
\hdim_G^\mathcal{S}(B)$;
the right-hand side merely gives a lower bound for the left-hand side.

Despite these difficulties the proof of the Interval
Theorem~\cite[Thm.~5.4]{KlThZR19} can be slightly modified to yield
the following convenient tool for investigating the Hausdorff spectra
of finite direct products of pro-$p$ groups.

\begin{Prop} \label{prop:KlThZR19} Let
  $G = H_1 \times \ldots \times H_r$ be a finite direct product of
  countably-based pro-$p$ groups $H_j$, equipped with a
  filration~$\mathcal{S} \colon G_i = H_{1,i} \times \ldots \times
  H_{r,i}$,
  $i \in \N_0$, that arises as the product of
  filtrations~$\mathcal{S}_j \colon H_{j,i}$, $i \in \N_0$, of~$H_j$
  for $1 \le j \le r$.

  Let $K \le_\mathrm{c} G$ be such that
  $K = K_1 \times \ldots \times K_r$ with $K_j \le_\mathrm{c} H_j$ for
  $1 \le j \le r$.  Suppose that $0 \le \varepsilon < \delta \le 1$ are
  such that, for each $j \in \{1,\ldots,r\}$, the group $K_j$ has
  strong Hausdorff dimension
  $\hdim_{H_j}^{\mathcal{S}_j}(K_j) = \delta$ and that all finitely
  generated subgroups $L \le_\mathrm{c} K_j$ have Hausdorff dimension
  $\hdim_{H_j}^{\mathcal{S}_j}(L) \le \varepsilon$.

  Then the Hausdorff spectrum of $G$ satisfies
  $(\varepsilon, \delta] \subseteq \hspec^{\mathcal{S}}(G)$.
\end{Prop}

Our strategy for proving Theorem~\ref{Thm:full-spectrum-Frattini} is
to apply this result in a situation, where the
$K_j \trianglelefteq_\mathrm{c} H_j$ are infinite normal subgroups of
infinite index, with Hausdorff dimension $\delta=1$, and such that
$\varepsilon$ can be taken to be~$0$.
	
\begin{Lem}\label{Lem:quotient_by_finite_subgroup_has_positive_rg}
  Let $G$ be a finitely generated infinite pro-$p$ group, and let
  $N \trianglelefteq_\mathrm{c} G$ be a finite normal subgroup.  Then
  $\RG(G/N) = \lvert N \rvert \,\RG(G)$.
\end{Lem}

\begin{proof}
  A filtration $\mathcal{S} \colon G_i$, $i \in \N$, of~$G$ induces a
  filtration $\mathcal{S} \vert_{G/N} \colon G_iN/N$, $i \in \N$, of
  $G/N$.  Since $N$ is finite, there exists $i_0 \in \N$ such that
  $G_i \cap N = 1$ for $i \geq i_0$.  This implies
  \[
  \RG(G/N) = \lim_{i \to \infty} \frac{\dd(G_iN/N) -1}{\lvert G/N : G_iN/N
    \rvert} =  \lim_{i \to \infty} \frac{\dd(G_i)-1}{\lvert G : G_iN
    \rvert} = \lvert N \rvert \, \RG(G). \qedhere
  \]
\end{proof}
	
We thank Mikhail Ershov for pointing out to us that the next result is
essentially proved in~\cite[Thm.~3.1]{BaSP13}.
	
\begin{Prop} \label{Prop:pos_rg_implies_not_ji} Let $G$ be a finitely
  generated pro-$p$ group of positive rank gradient.  Then $G$ admits
  an infinite normal subgroup $N \trianglelefteq_\mathrm{c} G$ such
  that $G/N$ has positive rank gradient.  Consequently, $G$ is not
  commensurable to any just infinite pro-$p$ group.
\end{Prop}

\begin{proof}
  Following~\cite{BaSP13}, for each normal subgroup
  $N \trianglelefteq_\mathrm{c} G$, we define $s(N,G)$ to be the
  infimum, over all (countable) sets $X$ of \emph{normal generators}
  of $N$ in $G$, of the quantity $\sum_{x \in X} p^{-\nu(x)}$, where
  $\nu(x) = \sup \{ k \in \N_0 \mid x \in G^{p^k} \}$ for each $x$.
  We observe that Theorem~3.1 in~\cite{BaSP13}, which is stated and
  proved for finitely generated discrete groups, translates almost
  verbatim to the pro-$p$ context:
  \[
  \RG(G/N) \geq \RG(G) - s(N,G).
  \]
		
  As $\RG(G) >0$, we find $k \in \N$ such that $\RG(G) - p^{-k} > 0$.
  As $G$ is infinite, so is the open normal subgroup~$G^{p^k}$, and we
  find among a finite generating set at least one element
  $x \in G^{p^k} \smallsetminus G^{p^{k+1}}$ whose normal closure
  $N = \langle x \rangle^G$ is infinite.  We observe that
  $s(N,G) = p^{-k}$ and conclude that
  \[
  \RG(G/N) \geq \RG(G) - s(N,G) > 0.
  \]
		
  Repeating the above procedure, we find an ascending chain
  $N = N_1 \lneqq N_2 \lneqq \ldots$ of normal subgroups of $G$ such
  that the factor groups $N_{i+1}/N_i$, $i \in \N$, are infinite.  In
  particular, $G$ has infinitely many commensurability classes of
  normal subgroups.  This implies that $G$ is not (abstractly)
  commensurable to a just infinite pro-$p$ group.
\end{proof}

\begin{Cor} \label{cor:has-inf-image} Let $G$ be a finitely generated
  pro-$p$ group of positive rank gradient.  Then $G$ has an infinite
  normal subgroup $N \trianglelefteq_\mathrm{c} G$ with
  $\lvert G : N \rvert = \infty.$
\end{Cor}

\begin{proof}
  Assume, for a contradiction, that $G$ has no infinite normal
  subgroup $N \trianglelefteq_\mathrm{c} G$ with
  $\lvert G : N \rvert = \infty$.  Using Zorn's Lemma we see that $G$
  has a just infinite quotient $G/N$ whose kernel $N$ must be finite.
  Lemma~\ref{Lem:quotient_by_finite_subgroup_has_positive_rg} implies
  that $G/N$ has positive rank gradient, in contradiction to
  Proposition~\ref{Prop:pos_rg_implies_not_ji}.
\end{proof}

\begin{proof}[Proof of Theorem~\ref{Thm:full-spectrum-Frattini}]
  Let $G = H_1 \times \ldots \times H_r$ be a non-trivial finite
  direct product of finitely generated pro-$p$ groups $H_i$ of
  positive rank gradient.  Observe that the Frattini
  series~$\mathcal{F}$ of~$G$ decomposes as the product of the
  Frattini series~$\mathcal{F}_j$ of the direct factors~$H_j$ for
  $1 \le j \le r$.

  By Corollary~\ref{cor:has-inf-image}, for each
  $j \in \{1,\ldots,r\}$, the group $H_j$ has an infinite normal
  subgroup $N_j \trianglelefteq_\mathrm{c} H_j$ with
  $\lvert G_j : N_j \rvert = \infty$.  From
  Proposition~\ref{Prop:normal_subgroups_hdim_1} we deduce that $N_j$
  has strong Hausdorff
  dimension~$\hdim_{H_j}^{\mathcal{F}_j} (N_j) = 1$; from
  Proposition~\ref{Prop:fg_subgroups_hdim_0} we conclude that,
  finitely generated subgroups $L \le_\mathrm{c} N_j$ have
  Hausdorff dimension~$0$ with respect to~$\mathcal{F}_j$.  Thus,
  Proposition~\ref{prop:KlThZR19} shows that $G$ has full Hausdorff
  spectrum.
\end{proof}


\section{Hausdorff spectrum with respect to iterated filtration
  series}\label{sec:verbal}

In this section we prove Theorem~\ref {thm:full_spec_iterated_verbal},
which generalises Theorem~\ref{Thm:full-spectrum-Frattini} for certain
groups to iterated verbal filtration series, defined in
Section~\ref{sec:introduction}; the key task is to extend
Propositions~\ref{Prop:fg_subgroups_hdim_0}
and~\ref{Prop:normal_subgroups_hdim_1}.
 
Let $G$ be a finitely generated pro-$p$ group equipped with a
filtration $\mathcal{S} \colon G_i$, $i \in \N_0$.  Let
$H \le_\mathrm{c} G$ be finitely generated with
$\lvert G : H \rvert = \infty$.  In order to prove that $H$ has strong
Hausdorff dimension $\hdim_G^\mathcal{S}(H) = 0$, in analogy to
Proposition~\ref{Prop:fg_subgroups_hdim_0}, it suffices to find
$i_0 \in \N$ and an unbounded monotone increasing sequence
$(m_i)_{i\geq i_0}$ such that
\begin{equation}\label{eqn:suff_cond_fg_sgp_hd_0}
  \vert G_i :  G_{i+1} \rvert \geq \lvert H \cap G_i :  H \cap G_{i+1}
  \rvert^{m_i} \qquad \text{for $i \ge i_0$.}
\end{equation}
For, if \eqref{eqn:suff_cond_fg_sgp_hd_0} holds, then we obtain
for $i \ge i_0$,
\[
\frac{\log_p \lvert H \cap G_{i_0} : H \cap G_i \rvert}{\log_p \lvert
  G_{i_0} : G_i \rvert} = \frac{\sum_{j = i_0}^{i-1} \log_p \lvert H
  \cap G_j : H \cap G_{j+1} \rvert}{\sum_{j = i_0}^{i-1} \log_p \lvert
  G_j : G_{j+1} \rvert} \leq \frac{1}{m_i},
\]
and hence
\begin{equation}\label{eqn:general_filtration_fg_sgp_hdim_0}
\hdim_G^\mathcal{S}(H) = \varliminf_{\substack{i\to \infty \\[2pt] i \ge i_0}}
  \frac{\log_p \lvert H : H \cap G_{i_0}
      \rvert + \log_p \lvert H \cap G_{i_0} : H
    \cap G_i \rvert}{\log_p \lvert G : G_{i_0}
      \rvert + \log_p \lvert G_{i_0} : G_i \rvert}
  \le \frac{1}{m_i} \xrightarrow[i\to \infty]{} 0.
\end{equation}
Moreover, the argument shows that $H$ has strong Hausdorff
dimension~$\hdim_G^{G_i}(H)=0$.
 
Similarly, if $N \trianglelefteq_\mathrm{c} G$ is an infinite normal
subgroup, we can prove that $N$ has strong Hausdorff dimension
$\hdim_G^\mathcal{S}(N) = 1$ in analogy to
Proposition~\ref{Prop:normal_subgroups_hdim_1}, once we have
identified $i_0 \in \N$ and an unbounded monotone increasing sequence
$(m_i)_{i\geq i_0}$ such that
\begin{equation} \label{eqn:suff_cond_normal_sgp_hd_1}
  \lvert G_i :  G_{i+1} \rvert \geq \lvert NG_i :  NG_{i+1}
  \rvert^{m_i} \qquad \text{for $i \geq i_0$.}
\end{equation}

As above, let $H \le_\mathrm{c} G$ be finitely generated with
$\lvert G : H \rvert = \infty$ and let
$N \trianglelefteq_\mathrm{c} G$ be infinite normal.  A natural way to
guarantee that conditions~\eqref{eqn:suff_cond_fg_sgp_hd_0} and
\eqref{eqn:suff_cond_normal_sgp_hd_1} hold is to identify an unbounded
increasing sequence $(m_i)_{i \in \N}$ such that
\begin{enumerate}[(a)]
\item $G_i$ maps homomorphically onto $(H\cap G_i)^{m_i}$,
  respectively $(NG_i)^{m_i}$, where
  \label{suff_conds_surject}
\item the image of $G_{i+1}$ is contained in
  $(H\cap G_{i+1})^{m_i}$, respectively $(NG_{i+1})^{m_i}$.
  \label{suff_conds_kernels_match}
\end{enumerate}

One way of ensuring that \eqref{suff_conds_kernels_match} holds is by
requiring that $G_{i+1}$ be a verbal subgroup of $G_i$, for each $i$.
Recall that a \emph{verbal subgroup} of a group $G$ is one generated
by all values $w(g_1,\ldots,g_r)$, for $g_1, \ldots, g_r \in G$, of
group words $w = w(X_1,\ldots,X_r)$ in a given set~$\mathcal{W}$.
Examples of filtrations satisfying this include the Frattini series,
which we already covered, and the \emph{iterated $p$-power series}
defined in Section~\ref{sec:introduction}.  One could even take
different sets $\mathcal{W}_i$ of words at each stage to define~$G_i$.

One way of ensuring that \eqref{suff_conds_surject} holds is, for
instance, by arranging that $G_i$ maps onto the free pro-$p$ group on
$m_i\dd(H\cap G_i)$, respectively $m_i\dd(NG_i)$, generators.  

\begin{proof}[Proof of Theorem~\ref{thm:full_spec_iterated_verbal}]
  The main task is to extend
  Propositions~\ref{Prop:fg_subgroups_hdim_0}
  and~\ref{Prop:normal_subgroups_hdim_1} to the situation, where (i)
  $G$ is either a finitely generated non-abelian free pro-$p$ group or
  a non-soluble Demushkin pro-$p$ group and (ii) the Hausdorff
  dimension function is associated to an iterated verbal filtration
  $\mathcal{S} \colon G_i$, $i \in \N_0$, of~$G$.  In particular, $G$
  has positive rank gradient.

  First suppose that $H \le_\mathrm{c} G$ is finitely generated and
  that~$\lvert G : H \rvert = \infty$.
  Lemma~\ref{Lem:ratio_gens_decreases} ensures that for each $n\in\N$,
  there exists $i_0 \in \N$ such that $n\dd(H\cap G_i) \le \dd(G_i)$
  for $i \ge i_0$.  Now, if $G$ is free, then so is each $G_i$; hence
  $G_i$ maps homomorphically onto the direct product of $n$ copies of
  $H \cap G_i$.  It is known that, if $G$ is a non-soluble Demushkin
  group, then so is~$G_i$; compare \cite{De61,An68,Se95}.  In this
  case $G_i$ maps homomorphically onto a free pro-$p$ group of rank
  $\lfloor \dd(G_i)/2 \rfloor$ and consequently onto the direct
  product of $\lfloor n/2 \rfloor -1$ copies of $H\cap G_i$; compare
  \cite[Thm.~12.3.1]{Wi98} and \cite{La67,Se95}.  Still considering
  $i\ge i_0$, write $m$ for the number of copies of $H\cap G_i$ that
  $G_i$ maps onto ($n$ in the first case, $\lfloor n/2 \rfloor -1$ in
  the second case). Since $G_{i+1}$ is a verbal subgroup of $G_{i}$,
  there is an induced epimorphism
  $G_i/G_{i+1} \to ((H\cap G_i)/(H\cap G_{i+1}))^m$.  Repeating this
  argument for $n \to \infty$, we construct an unbounded monotone
  increasing sequence $(m_i)_{i \ge i_0}$ such that
  \eqref{eqn:suff_cond_fg_sgp_hd_0} holds.  Proceeding as
  in~\eqref{eqn:general_filtration_fg_sgp_hdim_0}, we obtain
  $\hdim_G^\mathcal{S}(H)=0$.
		
  Next suppose that $N \trianglelefteq_\mathrm{c} G$ is infinite.
  Arguing very similarly, we obtain $i_0 \in \N$ and an unbounded
  monotone increasing sequence $(m_i)_{i\geq i_0}$ such that
  \eqref{eqn:suff_cond_normal_sgp_hd_1} holds.  From this we derive
  that $\hdim_G^\mathcal{S}(N) = 1$.

  Now let $G = E_1 \times \ldots \times E_r$ be a non-trivial finite
  direct product of finitely generated pro-$p$ groups $E_j$, each of
  which is either free or Demushkin, and let $G$ be equipped with an
  iterated verbal filtration $\mathcal{S} \colon G_i$, $i \in \N_0$.
  Suppose further that $E_1$ is non-soluble.  In order to deduce that
  $\hspec^{\mathcal{S}}(G) = [0,1]$ from
  Proposition~\ref{prop:KlThZR19}, as in the proof of
  Theorem~\ref{Thm:full-spectrum-Frattini}, we only need to explain
  how to deal with possible soluble direct factors~$E_j$ in the
  decomposition of~$G$.  Indeed, let $E = E_j$ be one of these soluble
  factors.  Then each of the groups $E \cap G_i$, $i \in \N_0$, is of
  the form $1$, $C_2$, $\Z_p$ or $\Z_p \ltimes \Z_p$ and can be
  generated by at most $2$~elements.  From this we deduce that
  \[
  \lim_{i \to \infty} \frac{\log_p \lvert E : E \cap G_i
    \rvert}{\log_p \lvert G : G_i \rvert} \le \lim_{i \to \infty}
  \frac{\log_p \lvert E : E \cap G_i \rvert}{\log_p \lvert E_1 : E_1
    \cap G_i \rvert} = 0,
  \]
  because $(E_1 \cap G_i)/(E_1 \cap G_{i+1})$ maps onto
  $((E \cap G_i)/(E \cap G_{i+1}))^{m_i}$ for an unbounded monotone
  increasing sequence $(m_i)_{i \ge i_0}$, by an argument very similar
  to the one given above for $H \le_\mathrm{c} G$.
\end{proof}

\begin{Rmk}
  (1) Perhaps it is not too surprising that our approach works for
  filtrations whose terms are recursively defined as verbal subgroups,
  because such filtrations grow at least as fast as the Frattini
  series.  To wit, it is not hard to show that every proper verbal
  subgroup of a finitely generated pro-$p$ group $G$ is contained in
  the Frattini subgroup $\Phi(G)$ of~$G$, because the elementary
  abelian $p$-group $G/\Phi(G)$ has no proper non-trivial verbal
  subgroups.

  (2) It is an interesting open problem to identify 
  further finitely generated pro-$p$ groups of positive rank
  gradient -- apart from free and Demushkin groups -- for which the
  conclusion of Theorem~\ref{thm:full_spec_iterated_verbal} holds,
  possibly even by a similar reasoning.
\end{Rmk}


\section{Non-abelian free pro-$p$ groups and the Zassenhaus
  series}\label{sec:Zassenhaus}
 
This section contains the proof of
Theorem~\ref{thm:full-spectrum-Zassenhaus-free} which relies on
Lie-theoretic techniques.  In the context of graded (restricted) Lie
algebras, the concept of density, which we recall in
Definition~\ref{defi:density}, is a natural counterpart to the notion
of Hausdorff dimension.

Results for free Lie algebras in~\cite{BaOl15,NeScSh00} imply that
non-zero ideals and finitely generated graded subalgebras
of free Lie algebras have density $1$ and~$0$,
respectively.  We establish analogous results for restricted Lie
algebras; see Propositions~\ref{prop:finitely_gen_subalg_density_0}
and~\ref{prop:subideals_density_1}.  A Lie-theoretic counterpart to
the Interval Theorem, stated at the end of
  Section~\ref{sec:introduction}, yields that free (restricted) Lie
algebras have full density spectrum; see
  Theorem~\ref{thm:all_densities_freelie} and
  Corollary~\ref{cor:all_densities_restricted}.  To obtain the full
Hausdorff spectrum for free pro-$p$ groups from that of free
restricted Lie algebras, we need to translate between
  subgroups of a free pro-$p$ group $F$ and subalgebras of the
  associated free restricted $\F_p$-Lie algebra.  This is possible due
  to consequences of a result of Ershov and
  Jaikin-Zapirain~\cite[Prop.~3.2]{ErJZ13}, which in turn relies on
  power-commutator factorisations in~$F$ resulting from Hall's
  commutator collection process.
 	
It would be very interesting to extend these methods so
that they cover other pro-$p$ groups resembling free pro-$p$
groups and their corresponding restricted $\F_p$-Lie algebras.
Already the case of one-relator groups poses challenges;
we deal with Demushkin pro-$p$ groups in
Section~\ref{sec:Demushkin}.

\begin{Not}
  Group commutators are written as $[x,y] = x^{-1}y^{-1}xy$.  Lie
  commutators are denoted by
  $[\mathsf{x},\mathsf{y}] = \mathsf{x}.\mathrm{ad}(\mathsf{y})$.
  Group and Lie commutators of higher degree are left-normed; for
  instance, $[x, y, z] = [[x,y], z]$ and
  $[\mathsf{x},\mathsf{y}, \mathsf{z}] = [[\mathsf{x},\mathsf{y}],
  \mathsf{z}]$.
  
  Additionally, we write $\langle X\rangle_\mathrm{Lie}$ respectively, $\langle X\rangle_\mathrm{res.\,Lie}$ for the Lie subalgebra, respectively, restricted Lie   subalgebra generated by a set $X$. We also write $\langle X\rangle_L$ for the (restricted) Lie ideal generated by $X$ in $L$.
\end{Not}

\subsection{Associated restricted Lie
  algebras} \label{subsec:assoc-rest-Lie} Recall from
Section~\ref{sec:introduction} the definition of the Zassenhaus series
$\mathcal{Z} \colon Z_i(G)$, $i \in \N$, of a finitely generated
pro-$p$ group~$G$.  It is the fastest descending series with
$[Z_i(G),Z_j(G)] \subseteq Z_{i+j}(G)$ and
$Z_i(G)^p \subseteq Z_{ip}(G)$ for all $i,j\geq 1$.
Lazard~\cite[Thm.~5.6]{La54} proved that
\[
Z_n(G)= \prod\nolimits_{j=0}^{\lceil \log_p(n) \rceil} \gamma_{\lceil
  n/p^j \rceil}(G) ^{\, p^j} \qquad \text{for $n \in \N$,}
\]
where $\gamma_i(G)$ denotes the $i$th term of the lower central series
of~$G$.  Each $Z_n(G)$ is open in $G$ and the Zassenhaus series
provides a filtration of~$G$; compare~\cite[Sec.~11.3]{DidSMaSe99}.
Note that each quotient $\mathsf{R}_n(G) = Z_{n}(G)/Z_{n+1}(G)$ is an
elementary abelian $p$-group, which can be regarded as an
$\F_p$-vector space.  The direct sum
\[
\mathsf{R}(G) =\bigoplus\nolimits_{n = 1}^\infty  \mathsf{R}_n(G)
\]
carries naturally the structure of a finitely generated graded
restricted $\F_p$-Lie algebra, with the Lie bracket and the $p$-map
induced from group commutators and the $p$-power map on
(representatives of) homogeneous elements;
compare~\cite[Sec.~12]{DidSMaSe99} or~\cite{Sh00}.  For convenience we
recall what this means and introduce the usual notation; see~\cite[\S
V.7]{Ja79} or \cite[\S 1.11]{Ba87}.
 
\begin{Defi}\label{Defi:restrictedLie}
  A \emph{restricted Lie algebra} (or \emph{Lie $p$-algebra})
  $\mathsf{R} = (\mathsf{R}, [\cdot,\cdot],
  {\cdot}^{[p]})$
  over the field~$\F_p$ consists of a $\F_p$-Lie algebra $\mathsf{R}$,
  with Lie bracket $[\cdot,\cdot]$, and a \emph{$p$-map}
  $\mathsf{R} \to \mathsf{R}$, $\mathsf{x} \mapsto \mathsf{x}^{[p]}$,
  satisfying the following properties:
  \begin{itemize}
  \item[(i)]
    $[\mathsf{x}, \mathsf{y}^{[p]}] =
    \mathsf{x}.\mathrm{ad}(\mathsf{y})^p$
    for $\mathsf{x},\mathsf{y} \in \mathsf{R}$,
  \item[(ii)] $(\alpha \mathsf{x})^{[p]} = \alpha \mathsf{x}^{[p]}$ for
    $\alpha \in \F_p$ and $\mathsf{x}\in \mathsf{R}$, \hfill (note that
    $\alpha^p = \alpha$ for $\alpha \in \F_p$)
  \item[(iii)]
    $(\mathsf{x} + \mathsf{y})^{[p]} = \mathsf{x}^{[p]} +
    \mathsf{y}^{[p]} + \sum\limits_{i=1}^{p-1}
    s_i(\mathsf{x},\mathsf{y})$
    for $\mathsf{x},\mathsf{y}\in R$, where
    $i \, s_i(\mathsf{x},\mathsf{y})$ is the coefficient of
    $\lambda^{i-1}$ in the generic expansion of
    $\mathsf{x}.(\mathrm{ad}(\lambda \mathsf{x} + \mathsf{y}))^{p-1}$.
  \end{itemize}
  The restricted Lie algebra $\mathsf{R}$ is ($\N$-)\emph{graded}, if
  $\mathsf{R} = \bigoplus_{n=1}^\infty \mathsf{R}_n$ decomposes as a
  direct sum of subspaces, which we refer to as
    \emph{homogeneous components}, such that
  $[\mathsf{R}_m,\mathsf{R}_n] \subseteq
  \mathsf{R}_{m+n}$
  and $\mathsf{R}_n^{\, [p]} \subseteq \mathsf{R}_{pn}$ for
  $m,n \in \N$.

  \emph{Restricted Lie subalgebras} and \emph{restricted Lie ideals}
  of $\mathsf{R}$ are defined in the usual way.  A restricted Lie
  subalgebra $\mathsf{S}$ of $\mathsf{R}$ forms a \emph{subideal} if
  there exists a finite sequence
  \[
  \mathsf{R} \trianglerighteq \mathsf{S}_1 \trianglerighteq
  \mathsf{S}_2 \trianglerighteq \ldots \trianglerighteq \mathsf{S}_n
  = \mathsf{S}
  \]
  of restricted Lie subalgebras of~$\mathsf{R}$, where each term is a
  Lie ideal in the preceding one.
\end{Defi}

\begin{Rmk}\label{rmk:Same_quotients_Zassenhaus}
  For every prime~$p$, the $p$-Zassenhaus series of a discrete
  group~$\Gamma$ is defined in the same way as for pro-$p$ groups and
  gives rise to a corresponding graded restricted $\F_p$-Lie algebra.
  If $\Gamma$ is finitely generated, with pro-$p$ completion
  $G = \widehat{\Gamma}_p$, the quotients $\Gamma/Z_n(\Gamma)$ are
  canonically isomorphic to $G/Z_n(G)$ for $n \in \N$.
  Furthermore, the associated restricted $\F_p$-Lie algebra
  $\mathsf{R}(\Gamma)$ is canonically isomorphic to
  $\mathsf{R}(G)$.
 	 
  If $\Gamma$ is free on $d \ge 2$ generators, then
  $G = \widehat{\Gamma}_p$ is a free pro-$p$ group on $d$ generators,
  and $\mathsf{R}(\Gamma) \cong \mathsf{R}(G)$ is the free restricted
  Lie algebra on $d$ generators, equipped with the natural grading;
  see~\cite[Thm.~6.5]{La54}.
\end{Rmk}

The (lower) density of subalgebras in graded restricted
$\F_p$-Lie algebras and in graded $\F_p$-Lie algebras is defined as
follows; see \cite{NeScSh00,BaOl15} for related notions, indeed
closely related for graded subalgebras.
  
\begin{Defi}\label{defi:density}
  Let $\mathsf{R} = \bigoplus_{n=1}^\infty \mathsf{R}_n$ be a
  non-zero graded restricted $\F_p$-Lie algebra such that
  each homogeneous component $\mathsf{R}_n$ is finite
  dimensional.  For each $n \in \N$, denote by
  $\mathsf{R}_{\ge n} = \bigoplus_{m = n}^\infty
   \mathsf{R}_m$
   the ideal of $\mathsf{R}$ consisting of all elements of degree at
   least~$n$, and set $\mathsf{R}_{> n} = \mathsf{R}_{\ge n+1}$.

  Let $\mathsf{S}$ be a restricted Lie subalgebra of $\mathsf{R}$.
  For each $n \in \N$, put
  \[
  \mathsf{S}_n = \big( (\mathsf{S} \cap \mathsf{R}_{\ge n}) +
  \mathsf{R}_{> n} \big) \cap \mathsf{R}_n
  \]
  so that
  $\mathsf{S}_\mathrm{grd} = \bigoplus_{n=1}^\infty \mathsf{S}_n$ is a
  graded restricted Lie subalgebra of~$\mathsf{R}$.

  The \emph{density} of $\mathsf{S}$ in
  $\mathsf{R}$ is defined as
  \[
  \dens_\mathsf{R}(\mathsf{S}) =
  \dens_\mathsf{R}(\mathsf{S}_\mathrm{grd}) = \varliminf_{n\to \infty}
  \frac{\dim_{\F_p} \big( (\mathsf{S} + \mathsf{R}_{>n} \big) /
    \mathsf{R}_{>n})}{\dim_{\F_p} (\mathsf{R} / \mathsf{R}_{>n})} =
  \varliminf_{n\to \infty} \frac{\sum_{m=1}^n\dim _{\F_p}
    (\mathsf{S}_m)}{\sum_{m=1}^n \dim _{\F_p} (\mathsf{R}_m)}.
  \]
  The \emph{density spectrum} of the graded restricted Lie algebra
  $\mathsf{R}$ is defined as
  \[
  \dspec(\mathsf{R}) = \{ \dens_\mathsf{R}(\mathsf{S}) \mid \text{$\mathsf{S}$ a
    restricted Lie subalgebra of $\mathsf{R}$} \}.
  \]

  Analogously, we define the density of Lie subalgebras $\mathsf{S}$
  in a graded $\F_p$-Lie algebra
  $\mathsf{L} = \bigoplus_{n=1}^\infty \mathsf{L}_n$, with finite
  dimensional homogeneous components $\mathsf{L}_n$, and
  we employ the same notation; in particular,
    $\mathsf{L}_{\ge n} = \bigoplus_{m = n}^\infty \mathsf{L}_m$ and
    $\mathsf{L}_{> n} = \mathsf{L}_{\ge n+1}$ for $n \in \mathbb{N}$.
\end{Defi}

In our setting, $G$ is a finitely generated infinite
pro-$p$ group with associated graded restricted $\F_p$-Lie
algebra~$\mathsf{R}(G)$.  The Hausdorff dimension of a subgroup
$H \le_\mathrm{c} G$ with respect to the Zassenhaus
series~$\mathcal{Z}$ of~$G$ coincides with the density of the
associated graded restricted Lie subalgebra
\[
\mathsf{R}_G(H) = \bigoplus_{n=1}^\infty \mathsf{R}_{G,n}(H), \quad
\text{where
  $\mathsf{R}_{G,n}(H) = \frac{(H \cap Z_n(G))
    Z_{n+1}(G)}{Z_{n+1}(G)}$ for $n \in \N$,}
\]
of~$\mathsf{R}(G)$, that is
$\hdim_G^\mathcal{Z}(H) = \dens_{\mathsf{R}(G)}(\mathsf{R}_G(H))$;
compare~\cite[\S 4.1]{Sh00}.

Note that, if $H$ is (sub)normal in~$G$, then $\mathsf{R}_G(H)$ is a
restricted Lie (sub)ideal in~$\mathsf{R}(G)$; in the converse
direction, if the restricted Lie subalgebra $\mathsf{R}_G(H)$ is
finitely generated, so is the subgroup~$H$.  However, in
  general the mapping $H \mapsto \mathsf{R}_G(H)$ from closed
subgroups of $G$ to restricted Lie subalgebras of~$\mathsf{R}(G)$ is
neither injective nor surjective.  Furthermore, the respective minimal
numbers of generators may change substantially; in particular, a
finitely generated subgroup need not produce a finitely generated
restricted Lie subalgebra.  Nor does a free subgroup need
  to produce a free restricted Lie subalgebra; compare
  Section~\ref{subsec:from-Lie_algebras_to_groups}.

In Section~\ref{subsec:density-spectrum} we compute the density
spectrum of free (restricted) Lie algebras.  In
Section~\ref{subsec:from-Lie_algebras_to_groups} we provide a bridge
to the Hausdorff spectrum of the corresponding free pro-$p$ groups; as
indicated, the transition requires some care.


\subsection{The density spectrum of finitely generated free
  (restricted) Lie algebras} \label{subsec:density-spectrum}

Let $d \in \N$ with $d \ge 2$.  In this section, $\mathsf{R}$ denotes
a free restricted $\F_p$-Lie algebra on $d$ generators
$\mathsf{x}_1, \ldots, \mathsf{x}_d$, and $\mathsf{L}$ denotes a free
$\F_p$-Lie algebra on $\mathsf{x}_1, \ldots, \mathsf{x}_d$ (embedded
as a Lie subalgebra into $\mathsf{R}$).  The Zassenhaus series
of~$\mathsf{R}$ -- or indeed, on a free pro\nobreakdash-$p$ group $F$
on $d$ generators, if we think of $\mathsf{R}$ as
$\mathsf{R} = \mathsf{R}(F)$ -- produces the natural grading
$\mathsf{R} = \bigoplus_{n=1}^\infty \mathsf{R}_n$ associated with Lie
words.  This induces a compatible grading
$\mathsf{L} = \bigoplus_{n=1}^\infty \mathsf{L}_n$ on~$\mathsf{L}$.

In fact, for all practical purposes, $\mathsf{R}$ and $\mathsf{L}$ can
be constructed as subalgebras of the free associative algebra

$\mathsf{A} = \F_p\langle \mathsf{x}_1, \ldots, \mathsf{x}_d
\rangle$,
turned into a restricted $\F_p$-Lie algebra via
$[\mathsf{u},\mathsf{v}] = \mathsf{u} \mathsf{v}
-\mathsf{v} \mathsf{u}$
and $\mathsf{u}^{[p]} = \mathsf{u}^p$ for
$\mathsf{u}, \mathsf{v} \in \mathsf{A}$.  From this perspective, the
homogeneous components $\mathsf{R}_n$ and $\mathsf{L}_n$ arise as
intersections of~$\mathsf{R}$, respectively~$\mathsf{L}$, with the
$\F_p$-subspace $\mathsf{A}_n$ spanned by all products of length $n$
in $\mathsf{x}_1, \ldots, \mathsf{x}_d$; compare~\cite[\S2]{Ba87}.

It is well-known that each homogeneous component $\mathsf{L}_n$ of
$\mathsf{L}$ is spanned, as an $\F_p$-vector space, by linearly
independent \emph{basic Lie commutators} in
$\mathsf{x}_1, \ldots, \mathsf{x}_d$ of weight~$n$; compare~\cite[\S
11]{Ha76} or \cite[\S 2]{Ba87}.  We recall the relevant notions
and record the analogous result for the free restricted
Lie algebra~$\mathsf{R}$.

\begin{Defi} \label{Defi:basic_comm_Lie} The \emph{weight} of a
  non-zero Lie commutator in a countable set $X$ of free
    generators is defined inductively: each generator
  $\mathsf{x} \in X$ has weight~$\wt(\mathsf{x}) = 1$,
  and, if $\mathsf{c}_1, \mathsf{c}_2$ are Lie commutators in $X$ such
  that $[\mathsf{c}_1, \mathsf{c}_2] \ne 0$, then
  $[\mathsf{c}_1, \mathsf{c}_2]$ is a Lie commutator of weight
  $\wt([\mathsf{c}_1,\mathsf{c}_2]) = \wt(\mathsf{c}_1) +
  \wt(\mathsf{c}_2)$.

  The \emph{$X$-basic commutators} are also defined
  inductively.  The $X$-basic commutators of weight $1$ are the
  generators contained in~$X$ in some well-order.  For
  each $n \in \N$ with $n \ge 2$, after defining $X$-basic commutators
  of weight less than~$n$, the $X$-basic commutators of weight~$n$ are
  those Lie commutators $[\mathsf{u},\mathsf{v}]$ such that:
  \begin{enumerate}
  \item $\mathsf{u},\mathsf{v}$ are $X$-basic commutators with
    $\wt(\mathsf{u})+\wt(\mathsf{v}) = n$,
  \item $\mathsf{u} > \mathsf{v}$; and if
    $\mathsf{u} = [\mathsf{y},\mathsf{z}]$ with $X$-basic commutators
    $\mathsf{y}, \mathsf{z}$, then $\mathsf{v} \geq \mathsf{z}$.
  \end{enumerate}
  Lastly, the well-order is extended, subject to the
  condition that $\mathsf{u} < \mathsf{v}$ whenever
  $\wt(\mathsf{u}) < \wt(\mathsf{v})$ and arbitrarily among the new
  $X$-basic commutators of weight~$n$. When the
    generating set $X$ is clear from the context, we speak of
    \emph{basic commutators}, and a similar terminology is used for
    group commutators.
\end{Defi}

For the following fundamental result we refer to~\cite[\S 2~Thm.~1 and \S
2~Prop.~14]{Ba87}.

\begin{Prop} \label{Prop:basis-restricted} Let
  $\mathsf{R} = \bigoplus_{n=1}^\infty \mathsf{R}_n$ be the free
  restricted $\F_p$-Lie algebra on
  $\mathsf{x}_1, \ldots, \mathsf{x}_d$, and let
  $\mathsf{L} \le \mathsf{R}$ be the free Lie algebra on
  $\mathsf{x}_1, \ldots, \mathsf{x}_d$, as above.  For each
  $n \in \N$, with $n = p^j m$ and $p \nmid m$, the $\F_p$-space
  $\mathsf{R}_n$ is spanned by the linearly independent elements
  \begin{equation*}
    \mathsf{c}^{[p]^i}, \quad \text{where $\mathsf{c} \in
      \mathsf{L}_{n/p^i}$ is a basic commutator for $0 \le i \le j$;}
  \end{equation*}  
  In particular,
  $\dim_{\F_p}(\mathsf{R}_n) = \sum_{i=0}^{j} \dim_{\F_p}(\mathsf{L}_{n/p^i})$.
\end{Prop}

Witt's formula gives
$\dim_{\F_p}(\mathsf{L}_n) = \tfrac{1}{n}\sum_{l \mid n}
\mu(l)d^{n/l}$,
where $\mu \colon \N \to \{0,1,-1\}$ denotes the classical M\"obius
function; see~\cite[Ex.~2 on p.~194]{Ja79} or \cite[\S 3
Thm.~1]{Ba87}.  This implies
\[
\dim_{\F_p}(\mathsf{R}_n) \\ = \frac{1}{p^j m} \sum_{l \mid m}
\mu(\nicefrac{m}{l})\left( d^{\, p^j l} + \sum\nolimits_{i=1}^j (p^i -
  p^{i-1}) \, d^{\, p^{j-i} l} \right),
\]
where $n = p^j m$ with $p \nmid m$ as in
Proposition~\ref{Prop:basis-restricted}.

We obtain simple, but significant consequences for the dimensions of
homogeneous components of $\mathsf{R}$ and~$\mathsf{L}$, including
their relation to one another; see~\cite[Lem.~4.3]{JZ08}.

\begin{Lem} \label{lem:densities_coincide_restricted_lie} In the setup
  described above,
  \begin{enumerate}[\rm (1)]
  \item $\dim_{\F_p}(\mathsf{R}_n) = \big( 1+o(1) \big) \frac{d^n}{n}$
    and $\dim_{\F_p}(\mathsf{L}_n) = \big( 1+o(1) \big) \frac{d^n}{n}$
    as $n \to \infty$;
  \item
    $\dim_{\F_p}(\mathsf{R} / \mathsf{R}_{>n}) =
    \big( 1 + o(1) \big) \frac{d^{n+1}}{n(d-1)}$
    and
    $\dim_{\F_p} (\mathsf{L} / \mathsf{L}_{> n}) =
    \big( 1 + o(1) \big) \frac{d^{n+1}}{n(d-1)}$ as $n \to \infty$.
  \end{enumerate}

  Consequently, for every graded Lie subalgebra
  $\mathsf{S} \le \mathsf{R}$ we have
  $\dens_\mathsf{R}(\mathsf{S}) = \dens_\mathsf{L}(\mathsf{S} \cap
  \mathsf{L})$.
\end{Lem}

\begin{proof}
  The estimates (1) and (2) are proved in~\cite[Lem.~4.3]{JZ08}.  We
  justify the conclusion for graded Lie
  subalgebras~$\mathsf{S} \le \mathsf{R}$.  For $m \in \N$ we write
  \[
  \mathsf{S}_m = \mathsf{S} \cap \mathsf{R}_m \qquad \text{and} \qquad
  (\mathsf{S} \cap \mathsf{L})_m = \mathsf{S} \cap \mathsf{L}_m =
  \mathsf{S}_m \cap \mathsf{L}.
  \]
  By~(2), we have
  \[
  \lim_{n\to \infty} \frac{\sum_{m=1}^n
    \dim_{\F_p}(\mathsf{R}_m)}{\sum_{m=1}^n\dim_{\F_p}(\mathsf{L}_m)}
  = 1.
  \]
  Clearly,  for $n \in \N$ the inequalities
  \[
  \sum_{m=1}^n \dim_{\F_p}(\mathsf{S}_m) - \sum_{m=1}^n
  (\dim_{\F_p}(\mathsf{R}_m) - \dim_{\F_p}(\mathsf{L}_m)) \le
  \sum_{m=1}^n \dim_{\F_p}((\mathsf{S} \cap \mathsf{L})_m) \le
  \sum_{m=1}^n \dim_{\F_p}(\mathsf{S}_m)
  \]
  hold.  Dividing by $\sum_{m=1}^n\dim_{\F_p}(\mathsf{R}_m)$ or,
  asymptotically equivalent, by
  $\sum_{m=1}^n\dim_{\F_p}(\mathsf{L}_m)$, and passing to the lower
  limit as $n \to \infty$, we obtain
  \[
  \dens_\mathsf{R}(\mathsf{S}) = \dens_\mathsf{R}(\mathsf{S}) - 0 \le
  \dens_\mathsf{L}(\mathsf{S} \cap \mathsf{L}) \le
  \dens_\mathsf{R}(\mathsf{S}). \qedhere
   \]
\end{proof}

The next lemma is elementary, but central for our constructions.

\begin{Lem} \label{lem:res-S-intersect-with-L} In the setup described
  above, let $\mathsf{M} \le \mathsf{L}$ be a Lie subalgebra.  Suppose
  that
  $\mathsf{M} = \langle \mathsf{y} \mid \mathsf{y} \in Y
  \rangle_{\F_p}$,
  as an $\F_p$-vector space, and let
  $\mathsf{S} = \langle \mathsf{M} \rangle_\mathrm{res.\, Lie} \le
  \mathsf{R}$
  be the restricted Lie subalgebra generated by~$\mathsf{M}$. Then we
  have
  \begin{equation}\label{equ:S-equlas-M-plus-Yspan}
    \mathsf{S} = \mathsf{M} + \langle \mathsf{y}^{[p]^j} \mid
    \mathsf{y} \in Y, j \in \N \rangle_{\F_p} 
    \qquad \text{and} \qquad \mathsf{S} \cap \mathsf{L} = \mathsf{M}.
  \end{equation}
  Moreover, if $\mathsf{M} \trianglelefteq \mathsf{L}$ is
    a Lie ideal, then
    $[\mathsf{S},\mathsf{R}] \subseteq \mathsf{M}$ and
    consequently $\mathsf{S} \trianglelefteq \mathsf{R}$ is a
    restricted Lie ideal.
\end{Lem}

\begin{proof}
  The first equation is easily verified, using the defining properties
  of the $p$-map.  For the second equation, we may choose
  $Y = \{ \mathsf{y}_1, \mathsf{y}_2, \ldots \}$ to be an $\F_p$-basis
  for $\mathsf{M}$ in the following way.  Each basis element
  $\mathsf{y}_i$ can be written as a linear combination of basic
  commutators for $\mathsf{L}$; let $\mathsf{c}_i$ denote
  the smallest basic commutator appearing in this expression, without
  loss of generality with coefficient~$1$.  Arrange that these
  smallest terms form an ascending chain
  $\mathsf{c}_1 < \mathsf{c}_2 < \ldots$ of basic commutators.  Using
  Proposition~\ref{Prop:basis-restricted} and working
  modulo~$\mathsf{L}$, we see that the elements $\mathsf{y}^{[p]^j}$,
  for $\mathsf{y} \in Y$ and $j \in \N$, are linearly independent.
  Indeed, it is enough to monitor the appearance of the linearly
  independent elements $\mathsf{c}_i^{\, [p]^j}$ in the decompositions
  of the elements in question, modulo~$\mathsf{L}$. In particular, we
  obtain
  \[
  \mathsf{S} = \mathsf{M} \oplus \langle \mathsf{y}^{[p]^j} \mid
  \mathsf{y} \in Y, j \in \N \rangle_{\F_p} \qquad \text{and} \qquad
  \mathsf{S} \cap \mathsf{L} = \mathsf{M}. 
  \]
    Finally, suppose that
    $\mathsf{M} \trianglelefteq \mathsf{L}$ is a Lie ideal. The
    description of $\mathsf{S}$ in \eqref{equ:S-equlas-M-plus-Yspan}
    and the defining properties of the $p$-map imply that
    $[\mathsf{S},\mathsf{R}] \subseteq \mathsf{M}$.
 \qedhere
\end{proof}

From~\cite[Thm.~4.1]{NeScSh00} we obtain the following useful
consequence.

\begin{Prop} \label{prop:fg-graded-0} Let $\mathsf{L}$ be a free
  non-abelian $\F_p$-Lie algebra on finitely many generators.  If
  $\mathsf{M} \lneqq \mathsf{L}$ is a finitely generated proper graded
  Lie subalgebra, then $\dens_\mathsf{L}(\mathsf{M}) = 0$.
\end{Prop}

\begin{Ex} \label{Ex:fg-proper-but-1} We provide an explicit example
  of a finitely generated proper Lie subalgebra
  $\mathsf{M} \lneqq \mathsf{L}$ that is not graded and
  has density $\dens_\mathsf{L}(\mathsf{M}) = 1$; this should be
  contrasted with \cite[Thm.~1]{BaOl15}.

  Suppose that $\mathsf{L}$ is free on~$\{ \mathsf{x}, \mathsf{y} \}$.
  Then the Lie subalgebra
  $\mathsf{M} = \langle \mathsf{x} +[\mathsf{x},
  \mathsf{y}], \mathsf{y} \rangle_\mathrm{Lie}$
  does not contain $\mathsf{x}$ and is thus properly contained
  in~$\mathsf{L}$.  Indeed, $\mathsf{L}$ maps, via
  $\mathsf{x} \mapsto 1$ and $\mathsf{y} \mapsto D$, onto the split
  extension $\F_p D \rightthreetimes \F_p[t]$ of the abelian Lie
  algebra $\F_p[t]$ by the $1$-dimensional Lie algebra $\F_p D$
  spanned by the derivation $D \colon \F_p[t] \to \F_p[t]$,
  $f(t) \mapsto tf(t)$.  The image of $\mathsf{M}$ under this
  epimorphism is the proper subalgebra
  $\F_p D \rightthreetimes (1+t)\F_p[t]$.

  On the other hand, the graded Lie algebra $\mathsf{M}_\mathrm{grd}$,
  constructed in Definition~\ref{defi:density}, is equal
  to~$\mathsf{L}$; hence $\dens_\mathsf{L}(\mathsf{M}) =1$.
\end{Ex}

Proposition~\ref{prop:fg-graded-0} implies, in particular, that proper
graded Lie subalgebras of finite codimension in $\mathsf{L}$ are never
finitely generated.  It is now easy to manufacture, for any given
$\alpha \in (0,1)$, a graded Lie subalgebra
$\mathsf{M} \le \mathsf{L}$ such that
$\dens_\mathsf{L}(\mathsf{M}) = \alpha$: one builds
$\mathsf{M} = \bigcup_{i=1}^\infty \mathsf{M}(i)$ as the union of an
ascending chain $\mathsf{M}(i)$, $i \in \N$, of suitably chosen
finitely generated graded Lie subalgebras, in analogy to the
construction underlying~\cite[Thm.~5.4]{KlThZR19}.  Below we provide a
detailed argument, following closely an analogous argument for pro-$p$
groups, given in \cite[Proof of Prop.~5.2]{KlThZR19}.

The purpose is to prepare the transition to a corresponding result for
the free restricted Lie algebra~$\mathsf{R}$.  This transition is not
straightforward, because an analogue of the Schreier formula implies
that every restricted Lie subalgebra $\mathsf{S}$ of finite
codimension in~$\mathsf{R}$ is finitely generated and thus constitutes
a potential obstacle to an analogous construction;
see~\cite[Thm.~2]{Ku72} and also~\cite{BrKoSt05} for the analogue of
Schreier's formula.

\begin{Thm}\label{thm:all_densities_freelie}
  Let $\mathsf{L}$ be a non-abelian free $\F_p$-Lie algebra on
  finitely many generators.  Then there exists, for each
  $\alpha \in [0,1]$, a Lie subalgebra $\mathsf{M} \leq \mathsf{L}$
  that can be generated freely by a suitable collection of basic
  commutators for $\mathsf{L}$ and
  satisfies~$\dens_\mathsf{L}(\mathsf{M}) = \alpha$.  In particular,
  $\dspec(\mathsf{L}) = [0,1]$.
\end{Thm}

\begin{proof}
  Recall that $\mathsf{L}$ is a free Lie algebra on generators
  $\mathsf{x}_1, \ldots, \mathsf{x}_d$, for $d \in \N$ with $d \ge 2$.
  Clearly, the trivial subalgebra $\{0\}$ and $\mathsf{L}$ have
  densities $0$ and $1$, respectively.

  Suppose now that $\alpha\in (0,1)$.  Observe that the sequence
  \[
  l(n) =\dim_{\F_p} (\mathsf{L} / \mathsf{L}_{>n}) = \sum\nolimits_{i=1}^n
  \dim_{\F_p}(\mathsf{L}_i), \quad n \in \N,
  \]
  is strictly increasing.  It suffices to produce a Lie subalgebra
  $\mathsf{M} = \bigoplus_i \mathsf{M}_i \leq \mathsf{L}$ which is
  generated by basic commutators such that
  \begin{enumerate}[\rm (i)]
  \item
    $\alpha- \nicefrac{1}{l(n)} \leq \tfrac{1}{l(n)} \sum_{i=1}^n
    \dim_{\F_p} (\mathsf{M}_i)$ for all $n \in \N$; and
  \item
    $\frac{1}{l(n)} \sum_{i=1}^n \dim_{\F_p} (\mathsf{M}_i) \leq
    \alpha$ for infinitely many $n \in \N$.
  \end{enumerate}
  Indeed, we can replace any generating set for $\mathsf{M}$
  consisting of basic commutators by an irredundant subset that is a
  free generating set for~$\mathsf{M}$; see~\cite[Proof of
  Thm.~2]{Sh09} and compare the introductory comments
  in~\cite{BrKoSt05}.

  We construct a Lie subalgebra $\mathsf{M}$ satisfying (i) and (ii)
  inductively as the union
  $\mathsf{M} = \bigcup_{k=1}^\infty \mathsf{M}(k)$ of an ascending
  chain of Lie subalgebras
  $\mathsf{M}(1) \subseteq \mathsf{M}(2) \subseteq \ldots$, where each
  term $\mathsf{M}(k) = \langle Y_k \rangle_\mathrm{Lie}$ is generated
  by a finite set $Y_k$ of basic commutators of weight at most $k$ and
  $Y_1 \subseteq Y_2 \subseteq \ldots$.

  Let $k \in \N$.  For $k=1$, choose $a \in \{0,1,\ldots, d-1\}$ such
  that $\alpha - \nicefrac{1}{d} \leq \nicefrac{a}{d} \leq \alpha$ and
  let $\mathsf{M}(1)$ denote the Lie subalgebra generated by
  $Y_1 = \{\mathsf{x}_1,\ldots, \mathsf{x}_a \}$.  Observe that
  $\mathsf{M}(1)$ is a proper finitely generated graded Lie subalgebra
  of~$\mathsf{L}$.

  Now suppose that $k \ge 2$.  Suppose further that we have already
  constructed a proper finitely generated Lie subalgebra
  $\mathsf{M}(k-1) = \langle Y_{k-1} \rangle_\mathrm{Lie} \le
  \mathsf{L}$
  which is generated by a finite set of basic commutators of weight at
  most $k-1$, and that the inequalities in (i) hold for
  $\mathsf{M}(k-1)$ in place of~$\mathsf{M}$ and $1 \le n \le k-1$.
  For $n \ge k$, consider
  \[
  \beta_n = \frac{\sum_{i=1}^n \dim_{\F_p} (\mathsf{M}(k-1)_i)}{l(n)}.
  \]
  
  First suppose that $\beta_k \le \alpha$.  Since
  \[
  \frac{\sum_{i=1}^{k-1} \dim_{\F_p} (\mathsf{M}(k-1)_i) +
    \dim_{\F_p}(\mathsf{L}_k)}{l(k)} \ge \left( \alpha -
    \tfrac{1}{l(k-1)} \right) \tfrac{l(k-1)}{l(k)} + \tfrac{l(k)-
    l(k-1)}{l(k)} \ge \alpha- \tfrac{1}{l(k)},
  \]
  we find a finite set
  $Z_k \subseteq \mathsf{L}_k \smallsetminus \mathsf{M}(k-1)_k$,
  consisting of basic commutators of weight $k$, such that
  \[
  \alpha- \tfrac{1}{l(k)} \leq \frac{\sum_{i=1}^{k-1} \dim_{\F_p}
    (\mathsf{M}(k-1)_i) + \dim_{\F_p} (\mathsf{M}(k-1)_k \oplus
    \langle Z_k\rangle_{\F_p})}{l(k)}\leq \alpha.
  \]
  Let $\mathsf{M}(k)$ be the Lie subalgebra generated by
  $Y_k = Y_{k-1} \cup Z_k$.  Observe that $\mathsf{M}(k)$ is a proper
  finitely generated graded Lie subalgebra of~$\mathsf{L}$ and that
  the inequality in~(i), respectively~(ii), holds for $\mathsf{M}(k)$
  in place of $\mathsf{M}$ and $1 \le n \le k$, respectively $n=k$.

  Now suppose that $\beta_k > \alpha$.  By
  Proposition~\ref{prop:fg-graded-0}, we find a minimal $k_0 \geq k+1$
  such that $\beta_{k_0} \le \alpha$.  Putting
  $\mathsf{M}(k) = \ldots = \mathsf{M}(k_0-1)$ and
  $Y_k = \ldots = Y_{k_0-1}$, we return to the previous case for
  $k_0>k$ in place of~$k$.
\end{proof}

The proof of Theorem~\ref{thm:all_densities_freelie} does not simply
carry over, mutatis mutandis, to free restricted Lie algebras, because
Proposition~\ref{prop:fg-graded-0} is lacking a direct analogue, as
explained before.  Nevertheless, we can deduce the analogous result as
follows.

\begin{Cor}\label{cor:all_densities_restricted}
  Let $\mathsf{R}$ be a non-abelian free restricted $\F_p$-Lie algebra
  on finitely many generators.  Then there exists, for each
  $\alpha\in[0,1]$, a graded restricted Lie subalgebra
  $\mathsf{S} \le \mathsf{R}$ that can be generated by basic
  commutators and satisfies $\dens_\mathsf{R}(\mathsf{S}) = \alpha$.
  In particular, $\dspec(\mathsf{R}) = [0,1]$.
\end{Cor}

\begin{proof}
  For $\alpha \in [0,1]$, Theorem~\ref{thm:all_densities_freelie}
  yields a Lie subalgebra $\mathsf{M} \leq \mathsf{L}$, generated by
  basic Lie commutators and such that
  $\dens_\mathsf{L}(\mathsf{M}) = \alpha$.  Let $\mathsf{S} = \langle
  \mathsf{M} \rangle_\mathrm{res.\,Lie}$ be the
  restricted Lie subalgebra of $\mathsf{R}$ that is generated by
  $\mathsf{M}$.  By Lemma~\ref{lem:res-S-intersect-with-L}, we have
  $\mathsf{S} \cap \mathsf{L} = \mathsf{M}$.  Hence
  Lemma~\ref{lem:densities_coincide_restricted_lie} implies that
  $\dens_\mathsf{R}(\mathsf{S}) = \dens_\mathsf{L}(\mathsf{M}) =
  \alpha$.
\end{proof}

It is interesting to give an alternative, more direct
proof of Corollary~\ref{cor:all_densities_restricted}, which we
proceed to do now.
 
\begin{Prop} \label{prop:subideals_density_1} If $\mathsf{S}$ is a
  non-zero restricted Lie subideal of the free restricted
  $\mathbb{F}_p$-Lie algebra~$\mathsf{R}$ then
  $\dens_\mathsf{R}(\mathsf{S}) = 1$.
\end{Prop}

\begin{proof}
  As
  $\dens_\mathsf{R}(\mathsf{S}) =
  \dens_\mathsf{R}(\mathsf{S}_\mathrm{grd})$,
  we may assume that $\mathsf{S}$ is graded.  Then
  $\mathsf{M} = \mathsf{S} \cap \mathsf{L}$ is a non-zero graded
  subideal of the free Lie algebra~$\mathsf{L}$.  By
  Lemma~\ref{lem:densities_coincide_restricted_lie} and
  \cite[Thm.~2]{BaOl15}, we have
  $\dens_\mathsf{R}(\mathsf{S}) = \dens_\mathsf{L}(\mathsf{M}) = 1$.
\end{proof}

\begin{Ex}
  We provide an explicit example of a finitely generated restricted
  Lie subalgebra $\mathsf{S} \le \mathsf{R}$ such that the Lie algebra
  $\mathsf{S} \cap \mathsf{L}$ is \emph{not} finitely generated.

  Indeed, suppose that $\mathsf{R}$ is free
  on~$\{ \mathsf{x}, \mathsf{y} \}$ and $\mathsf{L} \le \mathsf{R}$,
  as before; i.e., $d=2$ and $\mathsf{x} = \mathsf{x}_1$,
    $\mathsf{y} = \mathsf{x}_2$. Then the restricted Lie subalgebra
    $\mathsf{S} = \langle \mathsf{x}, \mathsf{y}^{[p]}
    \rangle_\mathrm{res.\, Lie} \le \mathsf{R}$
    induces the Lie subalgebra
  \[
  \mathsf{S} \cap \mathsf{L} =\langle [\mathsf{x}, \mathsf{y},
  \overset{ip}{\ldots}, \mathsf{y}] \mid i \geq 0
  \rangle_\mathrm{Lie} \le \mathsf{L}
  \]
  which is not finitely generated. Observe that
  $\mathsf{L} = \langle \mathsf{x}, \mathsf{y} \rangle_\mathrm{Lie}$
  maps onto the split extension $\F_p D \rightthreetimes \F_p[t]$ as
  described in Example~\ref{Ex:fg-proper-but-1}.  The image of
  $\mathsf{S} \cap \mathsf{L}$ under this epimorphism is the infinite
  dimensional abelian Lie algebra $\F_p[t^p]$.
\end{Ex}

The following proposition can be seen as a strong analogue of
Proposition~\ref{prop:fg-graded-0}.

\begin{Prop}\label{prop:finitely_gen_subalg_density_0}
  Let $\mathsf{S} \le \mathsf{R}$ be a finitely generated restricted
  Lie subalgebra of the free restricted $\mathbb{F}_p$-Lie algebra
  $\mathsf{R}$ such that
  $\dim_{\F_p}(\mathsf{R}/\mathsf{S}) = \infty$.  Then $\mathsf{S}$
  has density $\dens_\mathsf{R}(\mathsf{S}) = 0$.
\end{Prop}

\begin{proof}
  By \cite[Thm.~3]{Ku72} (also refer to~\cite{BrKoSt05}), the
  restricted Lie subalgebra $\mathsf{S}$ is a free factor of a
  restricted Lie subalgebra $\mathsf{U}$ of finite codimension
  in~$\mathsf{R}$.  Then $\mathsf{U} = \mathsf{S} \ast \mathsf{T}$ is
  a split extension
  $\mathsf{U} = \mathsf{S} \rightthreetimes \mathsf{I}$ of
  $\mathsf{I} = \ker(\eta)$ by~$\mathsf{S}$, where
  $\eta \colon \mathsf{U} \rightarrow \mathsf{S}$ denotes the retract
  that restricts to the identity on $\mathsf{S}$ and maps $\mathsf{T}$
  to~$\{0\}$.  Since $\mathsf{S}$ has infinite codimension
  in~$\mathsf{R}$, we deduce that $\mathsf{I} \ne \{0\}$.  By
  \cite[Lem.~10]{Ku83}, the restricted Lie algebra $\mathsf{U}$
  contains a restricted Lie ideal
  $\mathsf{J} \trianglelefteq \mathsf{R}$ of finite codimension.
  Hence $\mathsf{I} \cap \mathsf{J} \ne \{0\}$ is a restricted Lie
  subideal of $\mathsf{R}$ and
  $\dens_\mathsf{R}(\mathsf{I} \cap \mathsf{ }J)=1$, by
  Proposition~\ref{prop:subideals_density_1}.  Thus
  $\dens_{\mathsf{R}}(\mathsf{I}) = 1$ and the claim follows from
  \[
  1 = \dens_{\mathsf{R}}(\mathsf{U}) = \dens_{\mathsf{R}}(\mathsf{S}
  \rightthreetimes I) = \dens_{\mathsf{R}}(\mathsf{S}) +
  \dens_{\mathsf{R}}(\mathsf{I}) =
  \dens_{\mathsf{R}}(\mathsf{S}) + 1. \qedhere
  \]
\end{proof}

Putting together Propositions~\ref{prop:subideals_density_1} and
\ref{prop:finitely_gen_subalg_density_0} and \cite[Lem.~10]{Ku83}, we
obtain the following.

\begin{Cor} \label{cor:equiv-for-fg-restr-Lie_subalg} Let $\mathsf{S}$
  be a finitely generated restricted Lie subalgebra of the free
  restricted $\mathbb{F}_p$-Lie algebra~$\mathsf{R}$. Then the
  following are equivalent:
  \begin{enumerate}[\rm (1)]
  \item $\mathsf{S}$ contains a non-trivial restricted Lie subideal of
    $\mathsf{R}$;
  \item $\mathsf{S}$ has finite codimension in $\mathsf{R}$;
  \item $\mathsf{S}$ has density $\dens_\mathsf{R}(\mathsf{S}) = 1$.
  \end{enumerate}
\end{Cor}

An alternative proof of Corollary~\ref{cor:all_densities_restricted}
can now be obtained by modifying the proof of
Theorem~\ref{thm:all_densities_freelie} as follows.  The restricted
Lie ideal
$\mathsf{N} = \langle Y_1 \rangle_\mathsf{R} \trianglelefteq
\mathsf{R}$
generated by $Y_1$ is of infinite codimension and has density $1$
in~$\mathsf{R}$.  For each $k \in \N$, let
$\mathsf{S}(k) \le \mathsf{R}$ denote instead the restricted Lie
subalgebra generated by $Y_k$ and replace occurrences of
$\mathsf{L}_k$ by $\mathsf{N}_k = \mathsf{N} \cap \mathsf{R}_k$ so
that $\mathsf{S}$ is constructed inside~$\mathsf{N}$.


\subsection{From free restricted Lie algebras to free pro-$p$
  groups} \label{subsec:from-Lie_algebras_to_groups} 

  Let $G$ be a finitely generated pro-$p$ group, with
  associated restricted $\F_p$-Lie algebra~$\mathsf{R}(G)$; see
  Section~\ref{subsec:assoc-rest-Lie}.  Let
  $\eta \colon G \to \mathsf{R}(G)$ denote the \emph{standard map},
  given by $1 \mapsto 0$ and
  \[
  g \mapsto g Z_{n+1}(G) \in \mathsf{R}_n(G), \quad \text{where
    $g \in Z_n(G) \smallsetminus Z_{n+1}(G)$ for suitable $n \in \N$,}
  \]

  We observed before that, in general, the mapping
  $H \mapsto \mathsf{R}_G(H)$ from closed subgroups of $G$ to
  restricted Lie subalgebras of~$\mathsf{R}(G)$ is not particularly
  well behaved.  Let us illustrate some unpleasant phenomena by giving
  two concrete examples.

  \begin{Ex}\label{ex:notfgnorfree}
    Consider the cyclotomic extension of the ring of $p$-adic
    integers, $\mathfrak{O} = \mathbb{Z}_p[\zeta]$, where $\zeta$
    denotes a primitive $p^2$th root of unity.  Recall that the
    cyclotomic field $\mathbb{Q}_p(\zeta)$ is a totally ramified
    extension of $\mathbb{Q}_p$ of degree~$r = \varphi(p^2) = (p-1)p$.
    Indeed, $\pi = \zeta -1$ is a uniformising element and
    \[
    \mathfrak{O}_+ = \mathbb{Z}_p \oplus \pi \mathbb{Z}_p \oplus
    \ldots \oplus \pi^{r-1} \mathbb{Z}_p \cong \mathbb{Z}_p^{\, r}.
  \]
  Furthermore, we have $\pi^r \mathfrak{O} = p\mathfrak{O}$.
  
  We consider the semidirect product $G = T \ltimes A$, where
  $ T = \langle s \rangle \cong \mathbb{Z}_p$,
  \[
  A = \langle a_0, \ldots, a_{r-1} \rangle \cong \mathfrak{O}_+ \qquad
  \text{via} \qquad \psi \colon A \to \mathfrak{O}, \quad
  \prod\nolimits_{i=0}^{r-1} a_i^{\, \ell_i} \mapsto
  \sum\nolimits_{i=0}^{r-1} \ell_i \pi^i
  \]
  and the action of $T$ on $A$ is given by
  $(b^s) \psi = b \psi \cdot \zeta$ for $b \in A$.  The group $G$ is a
  torsion-free, metabelian $p$-adic analytic pro-$p$ group of
  dimension $r+1$.

  The first terms of the Zassenhaus series of~$G$ can be computed with
  ease; a description of the entire series would be cumbersome.  We
  obtain, for $1 \le n \le p^2$,
  \[
  Z_n(G) = \langle s^{\, p^{\lambda(n)}} \rangle \ltimes A_n, \quad
  \text{where $\lambda(n) = \lceil \log_p(n) \rceil$ and $A_n \psi =
    \pi^{n-1} \mathfrak{O}$.} 
  \]
  From this it is easy to check that in terms of the standard map
  $\eta \colon G \to \mathsf{R}(G)$ we have, for $1 \le n \le r$,
  \[
  \mathsf{R}_n(G)  =
  \begin{cases}
    \F_p (s \eta) + \F_p (a_0 \eta) & \text{is
      $2$-dimensional if $n = 1$,} \\
    \F_p (s \eta)^{[p]} + \F_p (a_{p-1} \eta) & \text{is
      $2$-dimensional if $n = p$,}\\
    \F_p (a_{n-1} \eta) & \text{is $1$-dimensional otherwise.}
  \end{cases}
  \]
  Furthermore, we obtain $(a_i \eta)^{[p]} = 0$ for $0 \le i \le p-2$.
  This implies, for instance, that the restricted Lie algebra
  $\mathsf{R}_G(H)$ corresponding to
  $H = \langle a_0 \rangle \cong \Z_p$ is neither $1$-generated nor
  free.
\end{Ex}

  \begin{Ex}
    The Sylow pro-$p$ subgroup $G = \mathrm{Aut}^1(\F_p(\!(t)\!))$ of
    $\mathrm{Aut}(\F_p(\!(t)\!))$, the automorphism group of the local
    field~$\F_p(\!(t)\!)$, consists of all
    $g \in \mathrm{Aut}(\F_p(\!(t)\!))$ such that
    $t.g \equiv_{t^2} t$.  The bijection $G \to t + t^2\F_p[\![t]\!]$,
    $g \mapsto t.g$ provides a realisation of $G$ as a group of formal
    power series under substitution; in this form the group is known
    as the Nottingham group over~$\F_p$; compare~\cite{Ca00}.  For
    simplicity we exclude the case $p=2$, which requires special
    attention.
  
    The group $G$ is a $2$-generated, hereditarily just infinite
    pro-$p$ group with many interesting properties; a base for the
    neighbourhoods of the identity element is given by the descending
    chain of open normal subgroups
    \[
    G_m = \{ g \in G \mid t.g \equiv_{t^{m+1}} t \}, \quad m \in \N.
    \]
    The depth of $g \in G$ is defined as
    $d(g) = \inf \{ m \in \N \mid g \not \in G_{m+1} \}$.
    Using~\cite[Rem.~1 and Lem.~1]{Ca00}, it is easy to see that for
    all $n \in \N$,
    \[
    Z_n(G) = \gamma_n(G) = G_{m(n)}, \qquad \text{where} \quad
      m(n) = n+1 + \lfloor (n-2)/(p-1) \rfloor.
    \]
    Conversely, this means that $G_m \subseteq Z_{n(m)}(G)$, but
    $G_m \not \subseteq Z_{n(m)+1}(G)$, where
    $n(m) = m-1 - \lfloor (m-2)/p \rfloor$.

    Consider the procyclic subgroup $H = \langle h \rangle \cong \Z_p$
    whose generator $h \in G$ is given by $t.h = t+t^3$.  Thus
    $d(g) = 2$ and, for $k \in \N_0$, it is known that
    $d_k = d(h^{p^k})$ satisfies $d_k \equiv_p d(h) = 2$ and
    $d_{k+1} \ge p d_k + 2$; compare~\cite[Lem.~1 and Thm.~6]{Ca00}.
    (In fact, $d_k = 2 (p^{k+1} -1)/(p-1)$ for all~$k \in \N_0$, in
    other words: $h$ is $2$-ramified; compare~\cite{No17}.)

    For $k \in \N_0$, we observe that
    $ n( d_k ) = \big( (p-1)d_k - (p-2) \big)/p$, hence
    \[
    n(d_{k+1}) \ge \frac{ (p-1) (p d_k +2) - (p-2) }{p} =
    (p-1) d_k + 1 > p \,  n(d_k),
    \]    
    and we conclude that $(h^{p^k} \eta)^{[p]} = 0$, where
    $\eta \colon G \to \mathsf{R}(G)$ denotes the standard map.  This
    shows that, restricted to $\mathsf{R}_G(H)$, the $p$-map is the
    null map.  Consequently, the infinite dimensional, abelian
    restricted Lie algebra $\mathsf{R}_G(H)$ is \emph{not} finitely
    generated.
  \end{Ex}

  For certain types of pro-$p$ groups~$G$, the mapping
  $H \mapsto \mathsf{R}_G(H)$ may display considerably better
  properties, but little in this direction seems to be known.  In the
  following we make use of a consequence that is easily derived from
  the proof of~\cite[Prop.~3.2]{ErJZ13}.  (Note that in the statement
  of~\cite[Prop.~3.2]{ErJZ13} it is assumed from the start that $F$ is
  free and that $W$ is a valuation; we consider situations that are, a
  priori, more general.)

  \begin{Prop} \label{pro:Andrei-Misha} Let $G$ be a finitely
    generated pro-$p$ group, and let $X \subseteq G$ be such that the
    standard map $\eta \colon G \to \mathsf{R}(G)$ induces a bijection
    from $X$ onto $Y = X \eta$.  Suppose
    that the restricted Lie subalgebra
    $\langle Y \rangle_\mathrm{res.\, Lie} \le \mathsf{R}(G)$ is
    freely generated by~$Y$.  

    Then the pro-$p$ group $H = \langle X \rangle \le_\mathrm{c} G$ is
    freely generated by the $1$-convergent subset~$X$, and
    $\mathsf{R}_G(H) = \langle Y \rangle_\mathrm{res.\, Lie}$.
\end{Prop}

\begin{proof}
  One easily verifies that $X$ is $1$-convergent.  Next we show that
  $H$ is freely generated by the $1$-convergent set~$X$;
  compare~\cite[\S 5.1]{Wi98}.  As explained in~\cite[\S 3.1]{ErJZ13}
  every element of $h \in H$ admits (possibly not uniquely) a
  power-commutator factorisation
  \begin{equation} \label{equ:p-c-factorisation}
  h = \prod_{(c,k) \in \mathcal{C} \times \N_0} c^{\,p^k  j_{c,k} },
  \end{equation}
  where $\mathcal{C}$ is the family (`multiset' in the parlance
  of~\cite{ErJZ13}) of $X$-basic group commutators and each $j_{c,k}$
  lies in $\{0,1,\ldots, p-1 \}$.  Tacitly, we have fixed an ordering
  of the countable family $\mathcal{C} \times \N_0$, i.e., a bijection
  from $\N$ to $\mathcal{C} \times \N_0$, so that we know in which
  order the infinite, but in any case convergent
  product~\eqref{equ:p-c-factorisation} is to be carried out.

  Furthermore, $H$ is freely generated by $X$ if and only if the
  neutral element $1$ has only one power-commutator factorisation,
  namely the trivial one resulting from choosing $j_{c,k} = 0$ for all
  $(c,k) \in \mathcal{C} \times \N_0$.  (Implicitly, this criterion or
  a similar one is also used for establishing the basic part of ``(ii) implies (i)''
  in~\cite[Prop.~3.2]{ErJZ13}.)  The mapping
  \[
  W \colon H \to [0,1], \quad h \mapsto 2^{-d(h)}, \qquad \text{where}
  \quad d(h) = \inf \{ n \in \N \mid h \in Z_n(G) \},
  \]
  is a pseudo-valuation on $H$, in the sense of~\cite{ErJZ13}.
  Adapting the argument that establishes ``(iii) implies (ii)''
  in~\cite[Prop.~3.2]{ErJZ13} (to the present situation $(H,W,X,Y)$ in
  place of $(F,W,X,S)$ in~\cite{ErJZ13}), we see that, indeed, any
  $h \in H$ that admits a non-trivial power-commutator factorisation
  cannot be equal to~$1$.  Thus $H$ is freely generated by~$X$.

  It remains to prove that
  $\mathsf{R}_G(H) = \langle Y \rangle_\mathrm{res.\, Lie}$.  Clearly,
  $\mathsf{R}_G(H) \supseteq \langle Y \rangle_\mathrm{res.\, Lie}$ holds, and
  the reverse inclusion is derived again from power-commutator
  factorisations, as in ``(iii) implies (iv)''
  in~\cite[Prop.~3.2]{ErJZ13}: briefly speaking,
  $h \in H \smallsetminus \{1\}$ with power-commutator
  factorisation~\eqref{equ:p-c-factorisation} is mapped under $\eta$
  to the finite linear combination
  \[
  h \eta = \sum_{(c,k) \in \mathcal{M}} j_{c,k} \big( c^\mathrm{Lie}
  \big)^{[p]^k}
  \]
  of iterated commutators in~$Y$, where
  $\mathcal{M} = \{ (c,k) \in \mathcal{C} \times \N_0 \mid p^k \wt(c)
  = d(h) \}$
  and $c^\mathrm{Lie}$ is the basic Lie commutator in~$Y$ that
  corresponds to the basic group commutator~$c$ in~$X$.
\end{proof}

  \begin{Cor} \label{cor:fg-H-gives-fg-Lie} Let $G$ be a finitely
    generated pro-$p$ group and let $H \le_\mathrm{c} G$ be a
    subgroup.  Suppose that the corresponding restricted Lie
    subalgebra $\mathsf{R}_G(H) \le \mathsf{R}(G)$ is free.  Then
    $\mathrm{d}(H)$, the minimal number of generators for~$H$, is
    equal to the free restricted rank of~$\mathsf{R}_G(H)$.  In
    particular, $H$ is finitely generated if and only if
    $\mathsf{R}_G(H)$ is.
\end{Cor}

\begin{proof}
  Suppose that
  $\mathsf{R}_G(H) = \langle Y \rangle_\mathrm{res. \, Lie}$ is freely
  generated by a set~$Y$ of homogeneous elements;
  compare~\cite[Lem.~2]{Ku72} and the comments in~\cite{BrKoSt05}.  We
  choose $X \subseteq H$ such that the standard map
  $\eta \colon G \to \mathsf{R}(G)$ induces a bijection from $X$
  onto~$Y$.  By Proposition~\ref{pro:Andrei-Misha}, the group
  $\langle X \rangle$ is freely generated by $X$ and
  $\mathsf{R}_G(\langle X \rangle) = \mathsf{R}_G(H)$.  Since
  $\langle X \rangle \le H$, this implies $\langle X \rangle = H$.
\end{proof}

\begin{proof}[Proof of Theorem~\ref{thm:full-spectrum-Zassenhaus-free}]
  Let $F = \langle x_1, \ldots, x_d \rangle$ be a finitely generated
  free pro-$p$ group of free rank~$d \ge 2$.  Then $\mathsf{R}(F)$ is
  a free restricted $\mathbb{F}_p$-Lie algebra and every restricted
  Lie subalgebra of $\mathsf{R}(F)$ is itself free.  Set
  $N = \langle x_2, \ldots, x_d \rangle^F \trianglelefteq_\mathrm{c}
  F$.
  Proposition~\ref{prop:subideals_density_1} shows that
  $\hdim_F^\mathcal{Z}(N) = 1$.  By
  Corollary~\ref{cor:fg-H-gives-fg-Lie} and
  Proposition~\ref{prop:finitely_gen_subalg_density_0}, every finitely
  generated subgroup $H \le_\mathrm{c} F$ with $H \subseteq N$ has
  Hausdorff dimension $\hdim_G^\mathcal{Z}(H) = 0$.  Thus the Interval
  Theorem, stated at the end of Section~\ref{sec:introduction}, shows
  that $\hspec^\mathcal{Z}(F) = [0,1]$.
\end{proof}


\section{Non-soluble Demushkin pro-$p$ groups and the Zassenhaus
  series} \label{sec:Demushkin}

In this section we extend the results obtained in
Section~\ref{subsec:from-Lie_algebras_to_groups} to cover non-soluble
Demushkin pro-$p$ groups and, more generally, mixed finite direct
products of finitely generated free and Demushkin pro-$p$ groups.  In
particular, we prove
Theorem~\ref{thm:full-spectrum-Zassenhaus-Demushkin} and
Corollary~\ref{cor:full-spectrum-Zassenhaus-mixed}.

Let $D$ be a non-soluble Demushkin pro-$p$ group.  According to the
classification of Demushkin groups (see~\cite[Sec.~12.3]{Wi98} and
\cite{La67,Se95}) the pro-$p$ group $D$ admits a finite one-relator
presentation of the form
\begin{equation} \label{equ:Demaushkin-presentation} D \cong \langle
  x_1, \ldots, x_d \mid r = r(x_1,\ldots,x_d) = 1 \rangle,
\end{equation}
where $d \ge 3$ and one of the following holds (in this context
$p^\infty$ is to be interpreted as~$0$):
\begin{enumerate}[$\circ$]
\item $d \equiv_2 0$ and
  $r = x_1^{\, p^f} [x_1,x_2] \cdots [x_{d-1},x_d]$ with
  $f \in \mathbb{N} \cup \{ \infty \}$ such that $p^f \ne 2$;
\item $d \equiv_2 1$, $p=2$ and
  $r = x_1^{\, 2} x_2^{\, 2^f} [x_2,x_3] [x_4,x_5] \cdots
  [x_{d-1},x_d]$ with $f \in \mathbb{N}_{ \ge 2} \cup \{ \infty \}$;
\item $d \equiv_2 0$, $p=2$, and
  $r = x_1^{\, 2+2^f} [x_1,x_2] [x_3,x_4] \cdots [x_{d-1},x_d]$ with
  $f \in \mathbb{N}_{\ge 2} \cup \{ \infty \}$
\item[] \hspace{0em}\phantom{$d \equiv_2 0$, $p=2$, and} or
  $r = x_1^{\, 2} [x_1,x_2] x_3^{\, 2^f} [x_3,x_4] \cdots
  [x_{d-1},x_d]$ with $f \in \mathbb{N}_{\ge 2}$.
\end{enumerate}


\subsection{Associated restricted Lie algebras and free subalgebras}
The graded restricted $\mathbb{F}_p$-Lie algebra associated to $D$
with respect to the Zassenhaus series has the form
\[
\mathsf{D} = \mathsf{R}(D) \cong \mathsf{R} / \mathsf{J} \qquad \text{for}
\qquad\mathsf{J} = \langle \mathsf{r} \rangle_\mathsf{R}
\trianglelefteq \mathsf{R},
\]
where $\mathsf{R}$ denotes the free restricted $\mathbb{F}_p$-Lie
algebra on $\mathsf{x}_1, \ldots, \mathsf{x}_d$ and
$\langle \mathsf{r} \rangle_\mathsf{R}$ denotes the restricted Lie
ideal generated by the homogeneous Lie element $\mathsf{r}$ given by
\begin{equation}
  \label{equ:generic-Demushkin-case}
  \mathsf{r} = [\mathsf{x}_1,\mathsf{x}_2] + \ldots +
  [\mathsf{x}_{d-1},\mathsf{x}_d], \qquad \text{with $d
    \ge 4$, $d \equiv_2 0$ and $p$ arbitrary,}
\end{equation}
or one of the following:
\begin{equation} \label{equ:special-Demushkin-cases}
  \begin{split}
    \mathsf{r} & = \mathsf{x}_1^{\, [2]} +
    [\mathsf{x}_2,\mathsf{x}_3] + \ldots +
    [\mathsf{x}_{d-1},\mathsf{x}_d] \qquad
    \text{with $d \ge 3$, $d \equiv_2 1$ and $p=2$;} \\
    \mathsf{r} & = \mathsf{x}_1^{\, [2]} +
    [\mathsf{x}_1,\mathsf{x}_2] + \ldots +
    [\mathsf{x}_{d-1},\mathsf{x}_d] \qquad \text{with
      $d \ge 4$, $d \equiv_2 0$ and $p=2$.}
  \end{split}
\end{equation}
Indeed, the presentation~\eqref{equ:Demaushkin-presentation} is
strongly free, as can be seen, for instance, from Anick's criterion;
compare \cite[\S 3]{Ga15} and \cite{La06,LaMi11}.  Applying
\cite[Thm.~2.12]{Ga15}, which is based on ideas put forward
in~\cite{La06}, we obtain the described presentation
for~$\mathsf{D}$.

Let $\mathsf{L}$ denote the free $\mathbb{F}_p$-Lie algebra on
$\mathsf{x}_1, \ldots, \mathsf{x}_d$, embedded in~$\mathsf{R}$, as in
Section~\ref{subsec:density-spectrum}.  For the next steps leading up
to Proposition~\ref{pro:Demushkin-Lie-has-free-density-1-subalgebra},
we focus on the `generic' case~\eqref{equ:generic-Demushkin-case},
where~$\mathsf{r} \in \mathsf{L}$ (and the specific shape of
$\mathsf{r}$ is of no further relevance).  In the proof of
Proposition~\ref{pro:Demushkin-Lie-has-free-density-1-subalgebra}, we
explain how the remaining `exceptional' cases
in~\eqref{equ:special-Demushkin-cases} can be reduced to the `generic'
case.  Assuming~\eqref{equ:generic-Demushkin-case}, we put
\[
\mathsf{I} = \langle \mathsf{r} \rangle_\mathsf{L},
\]
the Lie ideal generated by $\mathsf{r}$ in~$\mathsf{L}$.  From
Lemma~\ref{lem:res-S-intersect-with-L} we deduce the following result.

\begin{Lem}
  In the setup described above,
  $\langle \mathsf{I} \rangle_\mathrm{res.\, Lie} = \mathsf{J}$ and
  $\mathsf{I} = \mathsf{L} \cap \mathsf{J}$.
\end{Lem}

By \cite [Thm.~3]{La95}, we find a (free) graded Lie subalgebra
$\mathsf{M} \le \mathsf{L}$ with $\mathsf{M} \cap \mathsf{I} = 0$ and
such that $\mathsf{M} + \mathsf{I}$ has finite co-dimension
in~$\mathsf{L}$.  Put
\[
\mathsf{S} = \langle \mathsf{M} \rangle_\mathrm{res.\,Lie} \le
\mathsf{R}.
\]

\begin{Lem}
  In the setup described above, $\mathsf{S} \cap \mathsf{J} = 0$.
  Consequently,
  $(\mathsf{S} + \mathsf{J}) / \mathsf{J} \le \mathsf{R} / \mathsf{J}
  \cong \mathsf{D}$ is a free restricted graded Lie subalgebra.
\end{Lem}

\begin{proof}
  Suppose that $\mathsf{h} \in \mathsf{S} \cap \mathsf{J}$.  Recalling
  that $\mathsf{S} = \langle \mathsf{M} \rangle_\mathrm{res.\,Lie}$ and
  $\mathsf{J} = \langle \mathsf{I} \rangle_\mathrm{res.\,Lie}$, we
  apply Lemma~\ref{lem:res-S-intersect-with-L} to write
  \[
  \mathsf{h} = \mathsf{a} + \sum\nolimits_i \lambda_i \mathsf{y}_i^{\,
    [p]^{k(i)}} = \mathsf{b} + \sum\nolimits_j \mu_j \mathsf{z}_j^{\,
    [p]^{l(j)}},
  \]
  where $i$ and $j$ run through unspecified finite index sets and
  \[
  \mathsf{a}, \mathsf{y}_i \in \mathsf{M}, \quad \mathsf{b},
  \mathsf{z}_j \in \mathsf{I}, \quad \lambda_i, \mu_j \in
  \mathbb{F}_p, \quad k(i), l(j) \in \mathbb{N}.
  \]
  We observe that for every $\mathsf{m} \in \mathsf{M}$,
  \[
  [\mathsf{h},\mathsf{m}] =
  \underbrace{[\mathsf{a},\mathsf{m}]}_{\in \mathsf{M}} +
  \sum\nolimits_i \lambda_i \underbrace{[\mathsf{y}_i^{\,
      [p]^{k(i)}},\mathsf{m}]}_{\in \mathsf{M}} =
  \underbrace{[\mathsf{b},\mathsf{m}]}_{\in \mathsf{I}} +
  \sum\nolimits_j \mu_j \underbrace{[\mathsf{z}_j^{\,
      [p]^{l(j)}},\mathsf{m}]}_{\in \mathsf{I}} \in \mathsf{M} \cap
  \mathsf{I}= \{0\},
  \]
  In view of~\eqref{equ:S-equlas-M-plus-Yspan}, this
  implies
  $\mathsf{h} \in \mathrm{Z}_\mathsf{S}(\mathsf{M}) =
  \mathrm{Z}(\mathsf{S}) =\{ 0 \}$.
\end{proof}

\begin{Lem} \label{lem:co-dim-bound} Let $\mathsf{K} \le \mathsf{L}$
  be a Lie subalgebra of finite codimension $t$, say, in the free Lie
  algebra~$\mathsf{L}$.  Then, for $n \in \mathbb{N}$, we have
  \[
  \dim_{\mathbb{F}_p} \!\big( \mathsf{R} / (\langle \mathsf{K}
  \rangle_\mathrm{res.\,Lie} + \mathsf{R}_{> n}) \big) \le t \lfloor
  \log_p n \rfloor.
  \]
\end{Lem}

\begin{proof}
  Write
  $\mathsf{L} = \mathsf{K} \oplus \langle \mathsf{z}_1, \ldots,
  \mathsf{z}_t \rangle_{\mathbb{F}_p}$,
  as an $\mathbb{F}_p$-vector space.  From
  Lemma~\ref{lem:res-S-intersect-with-L} we conclude that
  \[
  \mathsf{R} = \langle \mathsf{L} \rangle_\mathrm{res.\,Lie} = \langle
  \mathsf{K} \rangle_\mathrm{res.\,Lie} + \langle \mathsf{z}_i^{\,
    [p]^j} \mid 1 \le i \le t, \; j \in \mathbb{N}
  \rangle_{\mathbb{F}_p}. \qedhere
  \] 
\end{proof}

\begin{Prop} \label{pro:Demushkin-Lie-has-free-density-1-subalgebra}
  The restricted $\mathbb{F}_p$-Lie algebra
  $\mathsf{D} = \mathsf{R}(D)$ contains a non-abelian free restricted
  graded Lie subalgebra $\mathsf{E}$ of
  density~$\dens_{\mathsf{D}}(\mathsf{E}) = 1$.
\end{Prop}

\begin{proof}
  First we continue to work in the setup that we built
  from~\eqref{equ:generic-Demushkin-case}.  We consider the free
  restricted Lie subalgebra $\mathsf{E}$ corresponding to
  $(\mathsf{S} + \mathsf{J}) / \mathsf{J} \le \mathsf{R} / \mathsf{J}
  \cong \mathsf{D}$.
  Let
  $t =
  \dim_{\mathbb{F}_p}(\mathsf{L}/(\mathsf{M}+\mathsf{I}))$.
  Observe that $\mathsf{x}_2$ and $\mathsf{x}_4$ generate a free Lie
  subalgebra of rank~$2$, modulo~$\mathsf{J}$.  Using
  Lemmata~\ref{lem:densities_coincide_restricted_lie}
  and~\ref{lem:co-dim-bound}, we obtain
  \begin{equation} \label{equ:E-has-hdim-1}
    \begin{split}
      \dens_{\mathsf{D}}(\mathsf{E}) %
      & = 1 - \varlimsup_{n \to \infty} \frac{\dim_{\F_p} \!\big(
        \mathsf{R} / (\mathsf{S} + \mathsf{J} + \mathsf{R}_{>n})
        \big)}{\dim_{\F_p} \!\big(\mathsf{R} / (\mathsf{J} + \mathsf{R}_{>n}) \big)} \\
      & \ge 1 - \varlimsup_{n \to \infty} \frac{t \lfloor \log_p n
        \rfloor}{n^{-1} 2^{n+1}} = 1.
    \end{split}
  \end{equation}

  It remains to deal with the `exceptional' cases
  in~\eqref{equ:special-Demushkin-cases}, which we ignored so far.  We
  show that these can be reduced to the `generic'
  case~\eqref{equ:generic-Demushkin-case}.

  Suppose that $\mathsf{D} \cong \mathsf{R} / \mathsf{J}$, where
  $\mathsf{J} = \langle \mathsf{r} \rangle_\mathsf{R}$ for
  $\mathsf{r} = \mathsf{x}_1^{\, [2]} +
  [\mathsf{x}_2,\mathsf{x}_3] + \ldots +
  [\mathsf{x}_{d-1},\mathsf{x}_d]$,
  with $d \equiv_2 1$ and $p=2$.  Then the restricted Lie subalgebra
  \[
  \widetilde{\mathsf{R}} = \langle \mathsf{x}_1^{\, [2]}, \mathsf{x}_2, \ldots,
  \mathsf{x}_d \rangle_\mathsf{R} = \langle \mathsf{x}_1^{\, [2]},
  \mathsf{x}_2, \ldots, \mathsf{x}_d, [\mathsf{x}_2,\mathsf{x}_1],
  \ldots, [\mathsf{x}_d,\mathsf{x}_1] \rangle_\mathrm{res.\,Lie}
  \]
  has co-dimension $1$ in~$\mathsf{R}$.  Moreover, writing
  $\mathsf{a} = \mathsf{x}_1^{\, [2]}$ and
  $\mathsf{b}_i = \mathsf{x}_i$,
  $\mathsf{c}_i = [\mathsf{x}_i,\mathsf{x}_1]$ for $2 \le i \le d$, we
  see that
  $\widetilde{\mathsf{R}} / \mathsf{J} = \widetilde{\mathsf{R}} / \langle
  \mathsf{r}, [\mathsf{r},\mathsf{x}_1] \rangle_{\widetilde{\mathsf{R}}}$
  has the finite presentation
  \begin{align*}
    \widetilde{\mathsf{R}} / \mathsf{J} %
    & \cong \langle \mathsf{a}, \mathsf{b}_2, \ldots, \mathsf{b}_d, \mathsf{c}_2,
      \ldots, \mathsf{c}_d \mid \mathsf{a} + [\mathsf{b}_2,\mathsf{b}_3] + \ldots +
      [\mathsf{b}_{d-1},\mathsf{b}_d],  \\
    &\omit \hfill $[\mathsf{b}_2,\mathsf{c}_3] +
      [\mathsf{c}_2,\mathsf{b}_3] + \ldots + [\mathsf{b}_{d-1}, 
      \mathsf{c}_d] + [\mathsf{c}_{d-1},\mathsf{b}_d]
      \rangle_\mathrm{res.\,Lie}$ \\ 
    & = \langle \mathsf{b}_2, \ldots, \mathsf{b}_d, \mathsf{c}_2,
      \ldots, \mathsf{c}_d \mid [\mathsf{b}_2,\mathsf{c}_3] +
      [\mathsf{c}_2,\mathsf{b}_3] + \ldots + [\mathsf{b}_{d-1},
      \mathsf{c}_d] +  [\mathsf{c}_{d-1},\mathsf{b}_d]
      \rangle_\mathrm{res.\,Lie}. 
  \end{align*} 
  Thus $\widetilde{\mathsf{R}}$ has a presentation of the
  form~\eqref{equ:generic-Demushkin-case} and, with small
  modifications, we can argue as in the `generic' case.  (For
  instance, in the special case $d=3$, we adapt the argument
  in~\eqref{equ:E-has-hdim-1} as follows: the elements $\mathsf{x}_1$
  and $\mathsf{x}_3$ generate, modulo~$\mathsf{J}$, a Lie subalgebra
  that maps, via $\mathsf{x}_1 + \mathsf{J} \mapsto 1$ and
  $\mathsf{x}_3 + \mathsf{J} \mapsto D$, onto the Lie algebra
  $\F_p D \rightthreetimes \F_p[t]$ discussed in
  Example~\ref{Ex:fg-proper-but-1}; this yields the considerably
  weaker, but still sufficient estimate
  $\dim_{\F_p} \!\big(\mathsf{R} / (\mathsf{J} + \mathsf{R}_{>n})
  \big) \ge n+1$ for $n \in \mathbb{N}$.)
 
  Finally, consider $\mathsf{D} \cong \mathsf{R} /\mathsf{J}$,
  where $\mathsf{J} = \langle \mathsf{r} \rangle_\mathsf{R}$ for
  $\mathsf{r} = \mathsf{x}_1^{\, [2]} + [\mathsf{x}_1,\mathsf{x}_2] +
  \ldots + [\mathsf{x}_{d-1},\mathsf{x}_d]$,
  with $d \equiv_2 0$ and $p=2$.  Then the restricted Lie subalgebra
  \[
  \mathsf{S} = \langle \mathsf{x}_1^{\, [2]}, \mathsf{x}_2, \ldots,
  \mathsf{x}_d \rangle_\mathsf{R} = \langle \mathsf{x}_1^{\, [2]},
  \mathsf{x}_2, \ldots, \mathsf{x}_d, [\mathsf{x}_2,\mathsf{x}_1],
  \ldots, [\mathsf{x}_d,\mathsf{x}_1] \rangle_\mathrm{res.\,Lie}
  \]
  has co-dimension $1$ in~$\mathsf{R}$.  Moreover, writing
  $\mathsf{a} = \mathsf{x}_1^{\, [2]}$, and
  $\mathsf{b}_i = \mathsf{x}_i$,
  $\mathsf{c}_i = [\mathsf{x}_i,\mathsf{x}_1]$ for $2 \le i \le d$,
  one sees that
  $\mathsf{S} / \mathsf{J} = \mathsf{S} / \langle \mathsf{r},
  [\mathsf{r},\mathsf{x}_1] \rangle_\mathsf{S}$
  has the finite presentation
  \begin{align*}
    \mathsf{S} / \mathsf{J} %
    & \cong \langle \mathsf{a}, \mathsf{b}_2, \ldots, \mathsf{b}_d,
      \mathsf{c}_2, \ldots, \mathsf{c}_d  \mid \mathsf{a}  -
      \mathsf{c}_2 + [\mathsf{b}_3,\mathsf{b}_4] + \ldots +
      [\mathsf{b}_{d-1},\mathsf{b}_d],  \\ 
    &\omit \hfill $[\mathsf{a},\mathsf{b}_2] +
      [\mathsf{b}_3,\mathsf{c}_4] + [\mathsf{c}_3,\mathsf{b}_4] +
      \ldots + [\mathsf{b}_{d-1}, 
      \mathsf{c}_d] + [\mathsf{c}_{d-1},\mathsf{b}_d] \rangle_\mathrm{res.\,Lie}$ \\
    & = \langle \mathsf{a}, \mathsf{b}_2, \ldots, \mathsf{b}_d,
      \mathsf{c}_3, \ldots, \mathsf{c}_d \mid 
      [\mathsf{a},\mathsf{b}_2] + [\mathsf{b}_3,\mathsf{c}_4] + \ldots +
      [\mathsf{c}_{d-1},\mathsf{b}_d] \rangle_\mathrm{res.\,Lie}.
  \end{align*} 
  Thus $\mathsf{S}$ has a presentation of the
  form~\eqref{equ:generic-Demushkin-case} and, again with small
  modifications, we can argue as in the `generic' case.
\end{proof}


\subsection{The Hausdorff spectrum}

We continue to consider the non-soluble Demushkin pro-$p$ group~$D$ on
$d$ generators, with presentation~\eqref{equ:Demaushkin-presentation},
and the associated restricted $\mathbb{F}_p$-Lie algebra
$\mathsf{D} = \mathsf{R}(D)$.  In the following we
  employ considerations that can be found in a very similar form, for
  instance, in~\cite{NeScSh00}, \cite[\S 3]{We15} or \cite[\S
  2]{MiRoTa16}.  But we need more careful estimates, going somewhat
  further than computing the entropy of relevant Lie algebras.  We
  adopt, with small changes, the notation employed
  in~\cite{MiRoTa16}.

\begin{Lem} \label{lem:Demushkin-growth-of-Dn}
  In the set-up described above, we have
  \[
  \dim_{\F_p}(\mathsf{D}_n) = \big( 1 + o(1) \big)
  \tfrac{(d-\varepsilon)^n}{n} \quad \text{and} \quad
  \dim_{\F_p}(\mathsf{D}/\mathsf{D}_{>n}) = \big( 1 + o(1) \big)
  \tfrac{(d-\varepsilon)^{n+1}}{n(d - \varepsilon-1)} \quad \text{as
    $n \to \infty$,}
  \]
  where $\varepsilon \in \mathbb{R}$ with $0 < \varepsilon < 1$
  satisfies $(d-\varepsilon) \varepsilon = 1$.
\end{Lem}

\begin{proof}
  Classical theorems of Jennings and Lazard (compare~\cite[\S
  12]{DidSMaSe99}) show that a universal restricted enveloping algebra
  of $\mathsf{D}$ is provided by the graded group algebra
  \[
  \mathrm{gr}(\mathbb{F}_p [D]) = \mathsf{A} = \bigoplus_{n=0}^\infty
  \mathsf{A}_n, \qquad \text{with $\mathsf{A}_n = \mathsf{I}^n /
    \mathsf{I}^{n+1}$ for $n \in \mathbb{N}_0$},
  \]  
  where $\mathsf{I}^n$ denotes the $n$th power of the augmentation
  ideal $\mathsf{I} = \mathrm{ker}(\mathbb{F}_p[D] \to \mathbb{F}_p)$.
  
  We consider the Hilbert--Poincar\'e series associated to the
  restricted Lie algebra~$\mathsf{D}$ and to the associative
  algebra~$\mathsf{A}$; they are defined as
  \[
  P_\mathsf{D}(t) = \sum _{n=1}^\infty c_n t^n \qquad
  \text{and} \qquad P_\mathsf{A}(t) = \sum _{n=0}^\infty a_n t^n,
  \]
  where
  \[
  c_n = \dim_{\mathbb{F}_p}(\mathsf{D}_n) \qquad \text{and} \qquad a_n =
  \dim_{\mathbb{F}_p}(\mathsf{A}_n).
  \]
  From Labute's work on mild pro-$p$ groups~\cite{La06} (see also
  \cite{LaMi11,Ga15}) it is known that
  \begin{equation} \label{equ:Hilbert-Poincare-Demushkin}
    \sum_{n=0}^\infty a_n t^n = 1 + \sum_{n=1}^\infty a_n t^n =
    \frac{1}{1-dt+t^2}.
  \end{equation}

  In accordance with \cite[Thm.~12.16]{DidSMaSe99}, we have the formal
  identity
  \[
  \sum_{n=0}^\infty a_n t^n = \prod_{n=1}^\infty \Big(
  \frac{1-t^{np}}{1-t^n} \Big)^{c_n}.
  \]
  Let $\mu \colon \mathbb{N} \to \{0,1,-1\}$ denote the classical
  M\"obius function, and write
  \[
  \sum_{n=1}^\infty b_n t^n = \log \Big( \sum_{n=0}^\infty a_n t^n
  \Big) = \sum_{n=1}^\infty c_n \log \Big( \frac{1-t^{np}}{1-t^n} \Big), 
  \]
  where $\log(1+X) = \sum_{k=1}^\infty (-1)^{k+1}k^{-1} X^k$ denotes the
  formal logarithm series.  As shown in \cite[Prop.~2.8]{MiRoTa16}
  this implies that, for $n = p^k m$ with $p \nmid m$,
  \[
  c_n = w_m + w_{pm} + \ldots + w_{p^k m},
  \]
  where the integer summands are given by the formula
  \[
  w_n = \frac{1}{n} \sum_{m \mid n} \mu(n/m) m b_m.
  \]

  We choose $\varepsilon \in \mathbb{R}_{>0}$ such that 
  \[
  1-dt+t^2 = \big( 1-(d-\varepsilon)t \big) \big( 1- \varepsilon t
  \big), \qquad \text{where $(d-\varepsilon) \varepsilon = 1$ and
    $d - \varepsilon > \varepsilon$.}
  \]
  From \eqref{equ:Hilbert-Poincare-Demushkin} we conclude that, for
  $n \in \mathbb{N}$,
  \[
  b_n = \frac{1}{n} \big( (d-\varepsilon)^n + \varepsilon^n \big)
  \]
  and
  \[
  w_n = \frac{1}{n} \sum_{m \mid n} \mu (n/m) \big( (d-\varepsilon)^m
  + \varepsilon^m \big).
  \]
  In particular,
  \[
  w_n \ge \frac{1}{n} (d-\varepsilon)^n - \frac{1}{n} \cdot n \cdot
  \big( 2 (d-\varepsilon)^{n/2} \big) = \left( 
 1 - \frac{2 n}{(d-\varepsilon)^{n/2}} \right) \frac{(d-\varepsilon)^n}{n} 
  \]
  is positive for all sufficiently large~$n$.
  Hence there is a constant $K \in \mathbb{N}$ such that
  \[
  c_n \ge w_n - K \ge \left( 1 - \frac{2 n}{(d-\varepsilon)^{n/2}}
  - \frac {K n}{(d-\varepsilon)^n} \right) \frac{(d-\varepsilon)^n}{n},
  \]
  while in the other direction we have, for similar reasons,
  \begin{align*}
    c_n %
    & \le \frac{1}{n} \big( (d-\varepsilon)^n + \varepsilon^n \big) +
      ( \log_p n )  \cdot \big( (d-\varepsilon)^{n/2}
      + \varepsilon^{n/2} \big) \\
    & = \Big( 1 + \left( \tfrac{\varepsilon}{d-\varepsilon} \right)^n
      + n ( \log_p n ) \left( \left( \tfrac{1}{d-\varepsilon}
      \right)^{n/2} + \left( \tfrac{\sqrt{\varepsilon}}{d-\varepsilon}
      \right)^n \right) \Big) \frac{(d-\varepsilon)^n}{n}
  \end{align*}
  These bounds yield the desired asymptotic estimate for
  $\dim_{\F_p}(\mathsf{D}_n)$.  The estimate for
  $\dim_{\F_p}(\mathsf{D}/\mathsf{D}_{>n})$ follows by standard manipulations;
  compare~\cite[Proof of Lem.~4.3]{JZ08}.
\end{proof}

\begin{Prop} \label{pro:fg-subalgs-in-Demushkin-dens-0} Let
  $\mathsf{S}$ be a finitely generated free restricted graded Lie
  subalgebra of~$\mathsf{D}$.  Then
  $\dens_{\mathsf{D}}(\mathsf{S}) = 0$.
\end{Prop}

\begin{proof}
  We continue to use the notation from the proof of
    Lemma~\ref{lem:Demushkin-growth-of-Dn}.  Now we consider the
  corresponding Hilbert--Poincar\'e series for the finitely generated
  free restricted graded Lie subalgebra $\mathsf{S} \le \mathsf{D}$
  and its universal enveloping algebra~$\mathsf{B}$.  Based on the
  `restricted Poincar\'e--Birkhoff--Witt Theorem', we
  take~$\mathsf{B}$ to be the graded subalgebra of $\mathsf{A}$ that
  is generated by the image of $\mathsf{S}$ under the canonical
  embedding~$\mathsf{D} \hookrightarrow \mathsf{A}$.  Clearly, we may
  assume that $\mathsf{S}$ is non-abelian and thus free on~$r \ge 2$
  generators.

  We write
  \[
  P_\mathsf{S}(t) = \sum _{n=1}^\infty \widetilde{c}_n t^n \qquad
  \text{and} \qquad P_\mathsf{B}(t) = \sum _{n=0}^\infty
  \widetilde{a}_n t^n,
  \]
  where
  \[
  \widetilde{c}_n = \dim_{\mathbb{F}_p}(\mathsf{S} \cap \mathsf{D}_n)
  \qquad \text{and} \qquad \widetilde{a}_n = \dim_{\mathbb{F}_p}
  (\mathsf{B} \cap \mathsf{A}_n).
  \]
  The universal restricted enveloping algebra $\mathsf{B}$ is a free
  associative $\mathbb{F}_p$-algebra in non-commuting homogeneous
  generators $\mathsf{y}_1, \ldots, \mathsf{y}_r$, say.  For each
  $i \in \mathbb{N}$, let 
  \[
  r_i = \lvert \{ s \mid 1 \le s \le r \text{ and } \mathsf{y}_s \in
  \mathsf{A}_i \} \rvert
  \] 
  denote the number of generators $\mathsf{y}_s$ of weight~$i$, and
  set $e = \max \{ i \in \mathbb{N} \mid r_i \ne 0 \}$.
  Then~\cite[Lem.~1.2]{An82} shows that
  \begin{equation}\label{equ:from-Anick}
    \sum_{n=0}^\infty \widetilde{a}_n t^n = 1 + \sum_{n=1}^\infty
    \widetilde{a}_n t^n = \frac{1}{1 - \sum_{i=1}^e r_i t^i} =
    \frac{1}{1- \lambda t} \prod_{i=1}^{e-1} \frac{1}{1- \nu_i t},
  \end{equation}
  where $\lambda \in \mathbb{R}_{> 1}$ and
  $\nu_i \in \mathbb{C} \smallsetminus \mathbb{R}$ with
  $\lvert \nu_i \rvert \le \lambda$ for $1 \le i < e$.  In analogy to
  $b_n$ and $w_n$, we consider
  \[
  \widetilde{b}_n = \frac{1}{n} \Big( \lambda^n + \sum_{i=1}^{e-1}
  \nu_i^{\, n} \Big) \qquad \text{and} \qquad \widetilde{w}_n =
  \frac{1}{n} \sum_{m \mid n} \mu (n/m) \Big( \lambda^m +
  \sum_{i=1}^{e-1} \nu_i^{\, m} \Big).
  \]
  Similar to the earlier situation we have, for $n = p^k m$ with
  $p \nmid m$,
  \[
  \widetilde{c}_n = \widetilde{w}_m + \widetilde{w}_{pm} + \ldots +
  \widetilde{w}_{p^k m};
  \]
  as a by-product we see that each $\widetilde{w}_n$ is an integer.
  We observe that
  \[
  \lambda^n \le \mathrm{Re} \Big( \lambda^n + \sum_{i=1}^{e-1} \nu_i^{\, n}
  \Big) \quad \text{for infinitely many $n \in \mathbb{N}$,}
  \]
  due to the fact that $\mathrm{Re}(\nu_1^{\, n}) , \ldots,
  \mathrm{Re}(\nu_{e-1}^{\, n}) \ge 0$ for infinitely many $n \in
  \mathbb{N}$ (compare \cite[Ch.~I, \S 5, Thm.~VI]{Ca57}),
  and that
  \[
  \Big\vert \lambda^n + \sum_{i=1}^{e-1} \nu_i^{\, n} \Big\vert \le
  e \lambda^n \quad \text{for all $n \in \mathbb{N}$.}
  \]
  This implies that, for infinitely many $n \in \mathbb{N}$,
  \[
  \widetilde{w}_n \ge \frac{1}{n}  \lambda^n -
  \frac{1}{n} \cdot n \cdot e \lambda^{n/2} = \big( 1 - e n
  \lambda^{-n/2} \big) \frac{\lambda^n}{n}
  \]
  and hence
  \[
  \widetilde{c}_n \ge \widetilde{w}_n - (\log_p n) \cdot e
  \lambda^{n/p} \ge \big( 1 - e n (1 + \log_p n) \lambda^{-n/2} \big)
  \frac{\lambda^n}{n},
  \]
  while, for all $n \in \mathbb{N}$,
  \begin{equation} \label{equ:c-tilde-klein} \widetilde{c}_n \le (
    \log_p n ) \cdot e \lambda^n = e ( \log_p n ) \lambda^n.
  \end{equation}

  Progressing along an infinite increasing sequence of natural
  numbers $n$ for which the estimates above are valid, we conclude
  that
  \[
  \big( 1 + o(1) \big) \frac{\lambda^n}{n} \le \widetilde{c}_n \le c_n
  = \big( 1 + o(1) \big) \frac{(d-\varepsilon)^n}{n} \qquad \text{as
    $n \to \infty$.}
  \]
  This implies that $\lambda \le d-\varepsilon$.  Next we claim that
  \begin{equation} \label{equ:lambda-klein}
  \lambda < d-\varepsilon.
  \end{equation}
  Indeed, the characteristic polynomial
  \[
  f = (t - \lambda) \prod_{i=1}^{e-1} (t - \nu_i) = t^k - \sum_{i=1}^k
  r_i t^{k-i} \in \mathbb{Z}[t]
  \]
  associated to the series in~\eqref{equ:from-Anick} satisfies
  $r_i \ge 0$ for $1 \le i \le e$.  Hence $f$ has precisely one
  positive real root, namely $\lambda$, and thus cannot be divisible
  by the polynomial $t^2-dt+1$ which has two real roots, namely
  $d-\varepsilon$ and $\varepsilon$.  But
  $t^2-dt+1$ is irreducible over $\mathbb{Z}$, as $d \ge 3$.  Thus
  $\lambda \ne d-\varepsilon$.

  From Lemma~\ref{lem:Demushkin-growth-of-Dn},
    \eqref{equ:c-tilde-klein} and \eqref{equ:lambda-klein} we deduce
  that
  \[
  \dens_\mathsf{D}(\mathsf{S}) = \varliminf_{n \to \infty}
  \frac{\sum_{m=1}^n \widetilde{c}_m}{\sum_{m=1}^n c_m} = 0. \qedhere
   \]
\end{proof}

\begin{proof}[Proof of
  Theorem~\ref{thm:full-spectrum-Zassenhaus-Demushkin}]
  From
  Propositions~\ref{pro:Demushkin-Lie-has-free-density-1-subalgebra}
  and~\ref{pro:Andrei-Misha} we conclude that there is a free subgroup
  $F \le_\mathrm{c} D$ such that $\mathsf{E} = \mathsf{R}_D(F)$ is a
  free restricted Lie subalgebra of $\mathsf{D} = \mathsf{R}(D)$ and
  $\hdim_D^\mathcal{Z}(F) = \dens_{\mathsf{D}}(\mathsf{E}) = 1$.  By
  Corollary~\ref{cor:fg-H-gives-fg-Lie} and
  Proposition~\ref{pro:fg-subalgs-in-Demushkin-dens-0}, every finitely
  generated subgroup $H \le_\mathrm{c} F$ has Hausdorff dimension
  $\hdim_D^\mathcal{Z}(H) = 0$.  Thus the Interval Theorem, stated at
  the end of Section~\ref{sec:introduction}, shows that
  $\hspec^\mathcal{Z}(F) = [0,1]$.
\end{proof}

\begin{proof}[Proof of
  Corollary~\ref{cor:full-spectrum-Zassenhaus-mixed}]
  From
  Lemmata~\ref{lem:densities_coincide_restricted_lie}
    and~\ref{lem:Demushkin-growth-of-Dn} it is easy to see that
  soluble factors can be neglected.  The assertion now follows
  directly from Proposition~\ref{prop:KlThZR19},
  Theorem~\ref{thm:full-spectrum-Zassenhaus-free} and
  Theorem~\ref{thm:full-spectrum-Zassenhaus-Demushkin}.
\end{proof}

\section{Normal and finitely generated Hausdorff
  spectra} \label{sec:normal-spectra}

In this section we prove Theorems~\ref{thm:direct_prods},
\ref{thm:direct_prods_free_Dem} and~\ref {thm:direct_prods_free}.
Recall from the introduction that the \emph{normal Hausdorff spectrum}
of a finitely generated pro-$p$ group~$G$, with respect to a
filtration~$\mathcal{S}$, is
$\hspec_\trianglelefteq^\mathcal{S}(G) = \{\hdim_G^\mathcal{S}(H) \mid
H \trianglelefteq_\mathrm{c} G\}$.

\begin{Lem} \label{lem:[N,H]} Let $H$ be a finitely generated pro-$p$
  group of positive rank gradient, and let
  $N \trianglelefteq_\mathrm{c} H$ be an infinite normal subgroup.
  Then,
  $[N,H] = \langle [x,y] \mid x \in N, \, y \in H \rangle
  \trianglelefteq_\mathrm{c} H$ is also infinite.
\end{Lem}

\begin{proof}
  Assume for a contradiction that $[N,H]$ is finite.  By
  Lemma~\ref{Lem:quotient_by_finite_subgroup_has_positive_rg},
  $G = H/[N,H]$ has positive rank gradient, and
  $N/[N,H] \subseteq \mathrm{Z}(G)$ shows that $G$ has infinite
  centre~$\mathrm{Z}(G)$.

  On the other hand, Theorem~\ref{Thm:posRG_normal_sgps_inf_gen}
  implies that $\mathrm{Z}(G)$ is locally finite.  For any finite
  central subgroup $K \trianglelefteq_\mathrm{c} G$, the rank gradient
  of $G/K$ is bounded by $d(G)-1$; thus
  Lemma~\ref{Lem:quotient_by_finite_subgroup_has_positive_rg} implies
  that $\mathrm{Z}(G)$ is finite; compare~\cite[Thm.~4]{AbJZNi11}.
\end{proof}

\begin{Prop} \label{prop:all-normals-direct-products} Let
  $G = H_1 \times \ldots \times H_r$ be a non-trivial direct product
  of finitely generated pro-$p$ groups $H_j$ of positive rank
  gradient, and let $N \trianglelefteq_\mathrm{c} G$.  Then there
  exist normal subgroups $K_j, N_j \trianglelefteq_\mathrm{c} H_j$,
  for $1 \le j \le r$, such that
  \[
  K_1 \times \ldots \times K_r \le N \le N_1 \times \ldots \times N_r
  \]
  and $K_j$ is infinite if and only if $N_j$ is infinite, for each
  $j \in \{1,\ldots,r\}$.
\end{Prop}

\begin{proof}
  For $1 \le j \le r$, let $N_j$ be the image of $N$ under the natural
  projection $G \to H_j$ and set $K_j = [N,H_j] = [N_j,H_j]$.  Now
  apply Lemma~\ref{lem:[N,H]}.
\end{proof}

\begin{proof}[Proof of Theorem~\ref{thm:direct_prods}]
  Consider first normal subgroups
  $N \trianglelefteq_\mathrm{c} G = H_1 \times \ldots \times H_r$ that
  decompose as $N = N_1 \times \ldots \times N_r$, with
  $N_j \trianglelefteq_\mathrm{c} H_j$ for $1 \le j \le r$.  Set
  $J = \{ j \mid 1 \le j \le r, \, \lvert N_j \rvert = \infty \}$.  A
  straightforward calculation, based on
  Proposition~\ref{Prop:normal_subgroups_hdim_1} and the fact that the
  the Frattini series~$\mathcal{F}$ of $G$ admits a direct product
  decomposition, yields
  $\hdim_G^\mathcal{F}(N) = \sum_{j \in J} \alpha_j$.  Finally, by
  Proposition~\ref{prop:all-normals-direct-products}, all normal
  subgroups of $G$ have Hausdorff dimension of the described form.
\end{proof}

\begin{proof}[Proof of Theorem~\ref{thm:direct_prods_free_Dem}]
  We argue as in the proof of Theorem~\ref{thm:direct_prods}, using
  the extension of Proposition~\ref{Prop:normal_subgroups_hdim_1} that
  we provided in the proof of
  Theorem~\ref{thm:full_spec_iterated_verbal}.
\end{proof}

\begin{proof}[Proof of Theorem~\ref{thm:direct_prods_free}]
  Clearly, we may assume that 
  \[
  d(F_j) = d \quad \text{for $1 \le j \le t$} \qquad \text{and}\qquad  d(F_j)
  < d \quad \text{for $t+1 \le j \le r$.}
  \]

  If $G$ is abelian, then $d=1$ and
  $F_1 \cong \ldots \cong F_t \cong \Z_p$, while the remaining factors
  are trivial.  The filtrations $\mathcal{S}$ and $\mathcal{Z}$
  consist of selected terms of the $p$-power series, possibly with
  repetitions.  It is easy to verify that closed (normal) subgroups
  $H \le_\mathrm{c} G$ have the expected Hausdorff dimension
  $\hdim_G^\mathcal{S}(H) = \hdim_G^\mathcal{Z}(H) = \dim(H) / \dim(G)$;
  compare~\cite[Thm.~1.1]{BaSh97}.

  Now suppose that $d \ge 2$, equivalently that $F_1$ is non-abelian.
  First we consider the iterated verbal filtration
  $\mathcal{S} \colon G_i$, $i \in \N_0$.  Suppose that, for
  $i \in \N$, the group $G_i$ is generated by all values of group
  words in a set~$\mathcal{W}_i$, say, and write
  $G_i = \mathcal{W}_i(G_{i-1})$ for short.  Let $1 \le j \le r$.  By
  Theorem~\ref{thm:direct_prods_free_Dem}, it suffices to prove that
  \begin{equation}\label{equ:alpha-j-claim}
    \alpha_j  = \hdim_G(F_j) =
    \begin{cases}
      \nicefrac{1}{t} & \text{if $1 \le j \le t$,} \\
      0 & \text{if $t < j \le r$.}
    \end{cases}
  \end{equation}

  Suppose that $t < k \le r$, thus $d = d(F_1) > d(F_k) = \tilde d$.
  We may regard $F_k = \langle x_1, \ldots, x_{\tilde d} \rangle$ as a
  finitely generated subgroup of infinite index
  in~$F_1 = \langle x_1, \ldots, x_d \rangle$.  Observe that, for
  $i \in \N$,
  \[
  F_k \cap \mathcal{W}_i ( \mathcal{W}_{i-1} ( \cdots (
  \mathcal{W}_1(F_1)) \cdots)) = \mathcal{W}_i ( \mathcal{W}_{i-1} (
  \cdots ( \mathcal{W}_1(F_k)) \cdots)).
  \]
  Hence the extension of Proposition~\ref{Prop:fg_subgroups_hdim_0}
  provided in the proof of Theorem~\ref{thm:full_spec_iterated_verbal}
  yields
  \[
  \lim_{i \to \infty} \frac{\log_p \lvert F_k : \mathcal{W}_i (
    \mathcal{W}_{i-1} ( \cdots ( \mathcal{W}_1(F_k)) \cdots))
    \rvert}{\log_p \lvert F_1 : \mathcal{W}_i ( \mathcal{W}_{i-1} (
    \cdots ( \mathcal{W}_1(F_1)) \cdots)) \rvert} = 0.
  \]
  Thus \eqref{equ:alpha-j-claim} follows from
  \begin{multline*}
    \hdim_G(F_j) \\
    = \varliminf_{ i\to \infty} \frac{\log_p \lvert F_j :
      \mathcal{W}_i ( \cdots ( \mathcal{W}_1(F_j)) \cdots) \rvert}{t
      \log_p \lvert F_1 : \mathcal{W}_i ( \cdots ( \mathcal{W}_1(F_1))
      \cdots) \rvert + \sum_{k=t+1}^r \log_p \lvert F_k :
      \mathcal{W}_i ( \cdots ( \mathcal{W}_1(F_k)) \cdots) \rvert}. 
  \end{multline*}

  Finally, we consider the Zassenhaus series~$\mathcal{Z}$.
  Lemma~\ref{lem:densities_coincide_restricted_lie} implies that, for
  $t < j \le r$,
  \begin{equation*} \label{equ:small-factors-irrelevant} \lim_{i \to
  	\infty} \frac{\log_p \lvert F_j : Z_i(F_j) \rvert}{\log_p \lvert
  	F_1 : Z_i(F_1) \rvert} = 0.
  \end{equation*}
  Suppose that
  $N \trianglelefteq_\mathrm{c} G = F_1\times\ldots\times F_r$
  decomposes as $N = N_1\times\ldots\times N_r$, and let
  $J=\{ j \mid 1 \le j \le t, \, \lvert N_j \rvert = \infty \}$.
  Propostion~\ref{prop:subideals_density_1} shows that, for
  $j \in J$,
  \begin{equation*} \label{equ:small-factors-irrelevant} \lim_{i \to
  	\infty} \frac{\log_p \lvert N_j Z_i(F_j) : Z_i(F_j) \rvert}{\log_p \lvert
  	F_1 : Z_i(F_1) \rvert} = 1.
  \end{equation*}
  Thus we obtain
  \begin{align*}
    \hdim_G^{\mathcal{Z}}(N) %
    & = \varliminf_{ i\to \infty} \frac{\sum_{j\in J}\log_p\lvert N_j Z_i(F_j)
      : Z_i(F_j) \rvert + \sum_{k\notin J} \log_p \lvert N_k Z_i(F_k)
      : Z_i(F_k) \rvert}{\sum_{j=1}^t \log_p \lvert F_j :
      Z_i(F_j) \rvert +
      \sum_{j=t+1}^r \log_p \lvert F_j : Z_i(F_j) \rvert}
      \\
    & = \varliminf_{ i\to \infty} \frac{\sum_{j\in J}\log_p\lvert N_j Z_i(F_j)
      : Z_i(F_j) \rvert + \sum_{k\notin J} \log_p \lvert N_k Z_i(F_k)
      : Z_i(F_k) \rvert}{t \log_p \lvert F_1 : Z_i(F_1) \rvert
      }  \\
    & = \frac{\lvert J\rvert}{t}.
  \end{align*}
  By Proposition~\ref{prop:all-normals-direct-products}, all normal
  subgroups of $G$ have Hausdorff dimension of this
  form.
\end{proof}

There is also a natural interest in the \emph{finitely generated
  Hausdorff spectrum} of a finitely generated pro-$p$ group~$G$, with
respect to a filtration~$\mathcal{S}$, defined as
\[
\hspec_\mathrm{fg}^\mathcal{S}(G) = \{\hdim_G^\mathcal{S}(H) \mid H
\le_\mathrm{c} G \text{ finitely generated} \};
\]
compare~\cite[\S 4.7]{Sh00} and \cite{KlThxx}.

Proposition~\ref{Prop:fg_subgroups_hdim_0}, %
its extension (given in the proof of Theorem~\ref{thm:full_spec_iterated_verbal}) and Corollary~\ref{cor:fg-H-gives-fg-Lie} %
show that
\[
\hspec_\mathrm{fg}^\mathcal{S}(G) = \{0,1\},
\]
provided that
\begin{enumerate}[$\circ$]
\item $G$ is a finitely generated pro-$p$ group of positive rank
  gradient, equipped with the Frattini series
  $\mathcal{S} = \mathcal{F}$, or
\item $G$ is a finitely generated non-abelian free pro-$p$ group or a
  non-soluble Demushkin pro-$p$ group, equipped with an iterated
  verbal filtration $\mathcal{S}$, or
\item $G$ is a finitely generated free pro-$p$ group, equipped with
  the Zassenhaus series $\mathcal{S} = \mathcal{Z}$.
\end{enumerate}

Surprisingly, the description of the finitely generated Hausdorff
spectra of finite direct products $G = F_1 \times \ldots \times F_r$
of non-abelian free pro-$p$ groups $F_j$, with respect to the Frattini
series~$\mathcal{F}$ (or more general iterated verbal filtrations),
appears to be more difficult than expected.  For simplicity, suppose
that $G = F \times F$ is the direct product of two copies of a free
pro-$p$ group $F$ with $d(F) = 2$.  It is easily seen that
$\hspec_\mathrm{fg}^\mathcal{F}(G) \supseteq \{0,\nicefrac{1}{2},1\}$.
However, the reverse inclusion does not seem to be clear, due to the
existence of `diagonal subgroups' of the form
\[
H_\varphi = \{ (x, x\varphi) \mid x \in U \},
\]
where $\varphi \colon U \to V$ is an isomorphism between open
subgroups $U, V \le_\mathrm{o} F$, necessarily of the same index
$\lvert F : U \rvert = \lvert F : V \rvert$, but otherwise
unrestricted.

\begin{Qu}
  Is it true that diagonal subgroups
  $H_\varphi \le_\mathrm{c} F \times F = G$ have Hausdorff dimension
  $\hdim_G^\mathcal{F}(H_\varphi) = \nicefrac{1}{2}$ with respect to
  the Frattini series~$\mathcal{F}$?
\end{Qu}



\begin{thebibliography}{10}

\bibitem {AbJZNi11}
M.~Ab\'{e}rt, A.~Jaikin-Zapirain, and N.~Nikolov.
\newblock The rank gradient from a combinatorial viewpoint.
\newblock {\em Groups Geom.\ Dyn.}, 5(2):213--230, 2011.

\bibitem{AbNi12}
M.~Ab\'{e}rt and N.~Nikolov.
\newblock Rank gradient, cost of groups and the rank versus {H}eegaard genus
  problem.
\newblock {\em J.\ Eur.\ Math.\ Soc.\ (JEMS)}, 14(5):1657--1677, 2012.

\bibitem{An68}
I.~V.~Andozhskii.
\newblock On subgroups of demushkin groups.
\newblock {\em Math.\ Notes of USSR}, 3(4):701--704, 1968.

\bibitem{An82}
D.~J.~Anick. 
\newblock Noncommutative graded algebras and their Hilbert series.
\newblock {\em  J.\ Algebra}, 78(1):120--140, 1982.

\bibitem{Ba87}
Yu.~A. Bahturin.
\newblock {\em Identical relations in {L}ie algebras}.
\newblock VNU Science Press, b.v., Utrecht, 1987.

\bibitem{BaOl15}
Yu.~Bahturin and A.~Olshanskii.
\newblock Growth of subalgebras and subideals in free Lie algebras.
\newblock {\em J.\ Algebra}, 422:277--305, 2015.

\bibitem{Ba04}
V.~G. Bardakov.
\newblock On a question of {D}. {I}. {M}oldavanski\u{\i} on the
  {$p$}-separability of subgroups of a free group.
\newblock {\em Siberian Math.\ J.}, 45(3):416--419, 2004.

\bibitem{BaSP13}
Y.~Barnea and J.~C. Schlage-Puchta.
\newblock On {$p$}-deficiency in groups.
\newblock {\em J.\ Group Theory}, 16(4):497--517, 2013.

\bibitem{BaSh97}
Y.~Barnea and A.~Shalev.
\newblock Hausdorff dimension, pro-{$p$} groups, and {K}ac-{M}oody algebras.
\newblock {\em Trans.\ Amer.\ Math.\ Soc.}, 349(12):5073--5091, 1997.

\bibitem{BrKoSt05}
R.~M.~Bryant, L.~G.~Kov\'acs, and R.~St\"ohr.
\newblock Subalgebras of free restricted Lie algebras.
\newblock {\em Bull.\ Austral.\ Math.\ Soc.}, 72(1):147--156, 2005. 

\bibitem{BuTh11}
J.~Button and A.~Thillaisundaram.
\newblock Applications of {$p$}-deficiency and {$p$}-largeness.
\newblock {\em Internat.\ J.\ Algebra Comput.}, 21(4):547--574, 2011.

\bibitem{Ca00}
R.~Camina.
\newblock The Nottingham group. 
\newblock in: {\em New horizons in pro-$p$ groups},
205--221, Birkh\"auser Boston, Boston, MA, 2000.

\bibitem{Ca57}
J.~W.~S.~Cassels. 
\newblock {\em An introduction to Diophantine approximation}. 
\newblock Cambridge University Press, New York, 1957.

\bibitem{De61}
S.~P.~Demushkin.
\newblock The group of a maximal $p$-extension of a local field.
\newblock {\em Izv.\ Akad.\ Nauk SSSR Ser.\ Mat.}, 25(3):329--346, 1961.

\bibitem{DidSMaSe99}
J.~D.~Dixon, M.~P.~F.~du~Sautoy, A.~Mann, and D.~Segal.
\newblock {\em Analytic pro-{$p$} groups}.
\newblock Cambridge University Press, Cambridge, second edition, 1999.

\bibitem{ErJZ13}
M.~Ershov and A.~Jaikin-Zapirain.
\newblock Groups of positive weighted deficiency and their applications.
\newblock {\em J.\ reine angew.\ Math.}, 677:71--134, 2013.

\bibitem{ErLu14}
M.~Ershov and W.~L\"{u}ck.
\newblock The first {$L^2$}-{B}etti number and approximation in arbitrary
  characteristic.
\newblock {\em Doc.\ Math.}, 19:313--332, 2014.

\bibitem{Fa14} 
K.~Falconer.  
\newblock {\em Fractal Geometry: Mathematical Foundations and
  Applications}.  
\newblock John Wiley \& Sons, Chichester, 2014.

\bibitem{Ga02}
D.~Gaboriau.
\newblock Invariants $l^2$ de relations d'\'equivalence et de groupes.
\newblock {\em Publ.\ Math.\ Inst.\ Hautes \'Etudes}, 95:93--150, 2002.

\bibitem{Ga15}
J.~G\"artner.
\newblock Higher Massey products in the cohomology of mild
pro-$p$-groups.
\newblock {\em  J.\ Algebra}, 422:788--820, 2015.

\bibitem{Gr60}
L.~Greenberg.
\newblock Discrete groups of motions. 
\newblock {\em Canad.\ J.\ Math.}, 12:414--425, 1960. 

\bibitem{Ha49}
M.~Hall, Jr.
\newblock Coset representations in free groups
\newblock {\em Trans.\ Amer.\ Math.\ Soc.} 47:421--432, 1949.

\bibitem{Ha76}
M.~Hall, Jr.
\newblock {\em The theory of groups}.
\newblock Chelsea Publishing Co., New York, 1976.

\bibitem{Ha58}
P.~Hall.
\newblock Some word-problems.
\newblock {\em J.\ London Math.\ Soc.}, 33(4):482--496, 1958.

\bibitem{Ja79}
N.~Jacobson.
\newblock {\em Lie algebras}.
\newblock Dover Publications, Inc., New York, 1979.

\bibitem{JZ08}
A.~Jaikin-{Z}apirain.
\newblock On the verbal width of finitely generated pro-{$p$} groups.
\newblock {\em Rev.\ Mat.\ Iberoam.}, 24:617--630, 2008.

\bibitem{KlThxx}
B.~Klopsch and A.~Thillaisundaram.
\newblock Normal Hausdorff spectra of pro-$p$ groups.
\newblock preprint (2018): \texttt{arXiv:1812.02322}.

\bibitem{KlThZR19}
B.~Klopsch, A.~Thillaisundaram, and A.~Zugadi-Reizabal.
\newblock Hausdorff dimensions in {$p$}-adic analytic groups.
\newblock {\em Israel J.\ Math.}, 231:1--23, 2019.

\bibitem{Ko13}
D.~H. Kochloukova.
\newblock Subdirect products of free pro-{$p$} and {D}emushkin groups.
\newblock {\em Internat.\ J.\ Algebra Comput.}, 23(5):1079--1098, 2013.

\bibitem{KoZa11}
D.~H. Kochloukova and P.~A. Zalesskii.
\newblock On pro-{$p$} analogues of limit groups via extensions of
  centralizers.
\newblock {\em Math.\ Z.}, 267(1-2):109--128, 2011.

\bibitem{Ku72}
G.~P. Kukin.
\newblock The subalgebras of free Lie $p$-algebras.
\newblock {\em Algebra and Logic}, 11(5):294--303, 1972.

\bibitem{Ku83}
G.~P. Kukin.
\newblock The equality problem and free products of Lie algebras and of
  associative algebras.
\newblock {\em Siberian Math.\ J.}, 24(2):221--231, 1983.

\bibitem{La67}
J.~P.~Labute.
\newblock Classification of Demushkin groups.
\newblock {\em Canad.\ J.\ Math.}, 19:106--132, 1967.

\bibitem{La95}
J.~P.~Labute.
\newblock Free ideals of one-relator graded Lie algebras. 
\newblock {\em Trans.\ Amer.\ Math.\ Soc.}, 347(1):175--188, 1995.

\bibitem{La06}
J.~Labute.
\newblock Mild pro-$p$-groups and Galois groups of $p$-extensions of
$\mathbb{Q}$.  
\newblock {\em J.\ reine angew.\ Math.}, 596(1):155--182, 2006.

\bibitem{LaMi11}
J.~Labute and J.~Min{\'a}{\v c}.
\newblock Mild pro-$2$-groups and $2$-extensions of $\mathbb{Q}$ with
restricted ramification.
\newblock  J.\ Algebra 332(1):136--158, 2011. 

\bibitem{La05} 
M.~Lackenby.  
\newblock Expanders, rank and graphs of groups.  
\newblock {\em Israel J.\ Math.}, 146(1):357--370, 2005.

\bibitem{La09}
M.~Lackenby.
\newblock Large groups, property {$(\tau)$} and the homology growth of
  subgroups.
\newblock {\em Math.\ Proc.\ Cambridge Philos.\ Soc.}, 146(3):625--648, 2009.

\bibitem{La10}
M.~Lackenby.
\newblock Detecting large groups.
\newblock {\em J.\ Algebra}, 324(10):2636--2657, 2010.

\bibitem{La54}
M.~Lazard.
\newblock Sur les groupes nilpotents et les anneaux de {L}ie.
\newblock {\em Ann.\ Sci.\ Ecole Norm.\ Sup.\ (3)}, 71(2):101--190, 1954.

\bibitem{Lu82}
A.~Lubotzky.
\newblock Combinatorial group theory of pro-$p$ groups.
\newblock {\em J.\ Pure Appl.\ Algebra}, 25(3):311--325, 1982.

\bibitem{MiRoTa16}
J.~Min{\' a}{\v c}, M.~Rogelstad, and N.~D.~T{\^ a}n.
\newblock Dimensions of Zassenhaus filtration subquotients of some
pro-$p$-groups.
\newblock {\em Israel J.\ Math.}, 212(2):825--855, 2016. 

\bibitem{NeScSh00}
M.~F.~Newman, C.~Schneider, and A.~Shalev.
\newblock The entropy of graded algebras.
\newblock {\em J.\ Algebra}, 223(1):85--100, 2000.

\bibitem{No17}
J.~Nordqvist.
\newblock Characterization of $2$-ramified power series.
\newblock {\em J.\ Number Theory}, 174:258--273, 2017. 

\bibitem{Os11}
D.~Osin.
\newblock Rank gradient and torsion groups.
\newblock {\em Bull.\ Lond.\ Math.\ Soc.}, 43(1):10--16, 2011.

\bibitem{Pa15}
N.~Pappas.
\newblock Rank gradient and {$p$}-gradient of amalgamated free products and
  {HNN} extensions.
\newblock {\em Comm. Algebra}, 43(10):4515--4527, 2015.

\bibitem{Ro70}
C.~A. Rogers.
\newblock {\em Hausdorff Measures}.
\newblock Cambridge University Press, London, 1970.

\bibitem{SP12}
J.~C.~Schlage-Puchta 
\newblock A {$p$}-group with positive rank gradient.
\newblock {\em J.\ Group Theory}, 15(1):261--270, 2012.

\bibitem{Se95}
J.-P.~Serre.
\newblock Structure de certains pro-$p$-groupes (d'apr\`es Demu\v{s}kin).
\newblock  In {\em S\'eminaire Bourbaki}, Vol.~8, Exp.\ No.\ 252, 145--155,
Soc.\ Math.\ France, Paris, 1995.

\bibitem{Sh00}
A.~Shalev.
\newblock Lie methods in the theory of pro-{$p$} groups.
\newblock In {\em New horizons in pro-{$p$} groups}, pages
1--54. Birkh\"auser Boston, Boston, MA, 2000. 

\bibitem{Sh09}
A.~I.~Shirshov.
\newblock Subalgebras of free Lie algebras.
\newblock In {\em Selected Works of A.~I.~Shirshov.}, pages 3--13.
Birkh\"auser, Basel, 2009.

\bibitem{Sh17}
M.~Shusterman.
\newblock Ascending chains of finitely generated subgroups.
\newblock {\em J.\ Algebra}, 471:240--250, 2017.

\bibitem{SnZa16}
I.~Snopce and P.~A.~Zalesskii.
\newblock Subgroup properties of {D}emushkin groups.
\newblock {\em Math.\ Proc.\ Cambridge Philos.\ Soc.}, 160(1):1--9, 2016.

\bibitem{We15}
Th.\ Weigel.
\newblock Graded Lie algebras of type $\mathsf{FP}$.
\newblock {\em Israel J.\ Math.}, 205(1):185--209, 2015.

\bibitem{Wi98}
J.~S.~Wilson.
\newblock {\em Profinite groups}.
\newblock Oxford University Press, New York, 1998.
\end{thebibliography}
\end{document}